\documentclass[11pt,reqno]{amsart}

\usepackage[text={165mm,245mm},centering]{geometry}
\usepackage{graphicx}
\usepackage{amssymb}
\usepackage{epstopdf}

\usepackage{latexsym}

\usepackage{amsmath,amsthm}

\usepackage[mathscr]{eucal}
\usepackage[all]{xy}
\usepackage{mathrsfs}
\usepackage{hyperref}

\newtheorem{theorem}{Theorem}
\newtheorem{lemma}[theorem]{Lemma}

\newtheorem{corollary}[theorem]{Corollary}
\newtheorem{proposition}[theorem]{Proposition}

\theoremstyle{definition}
\newtheorem{example}[theorem]{Example}
\newtheorem{definition}[theorem]{Definition}
\newtheorem{definition-lemma}[theorem]{Definition-Lemma}
\newtheorem{definition-theorem}[theorem]{Definition-Theorem}
\newtheorem{remark}[theorem]{Remark}

\newtheorem*{ack}{Acknowledgements}

\newcommand{\ac}{\textup{!`}}

\newcommand{\ot}{\otimes}

\renewcommand{\k}{\Bbbk}

\newcommand{\lra}{\longrightarrow}

\newcommand{\ra}{\rightarrow}
\newcommand{\sdp}{\times\kern-.2em\vrule height1.1ex depth-.05ex}
\newcommand{\epi}{\lra \kern-.8em\ra}

\newcommand{\ol}{\overline}

\newcommand{\hdot}{\;\raisebox{3.2pt}{\text{\circle*{2.5}}}}

\allowdisplaybreaks

\title[Koszul Calabi-Yau algebras]
{Batalin-Vilkovisky algebras and the noncommutative Poincar\'e duality of Koszul Calabi-Yau algebras}

\author{Xiaojun Chen, Song Yang}

\address{Department of Mathematics, Sichuan University, Chengdu,
Sichuan, 610064,  People's Republic of China}

\email{xjchen@scu.edu.cn, syang.math@gmail.com}

\author{Guodong Zhou}

\address{Department of Mathematics, Shanghai Key laboratory of PMMP,
East China Normal  University, Shanghai,
200241,  People's Republic of China}

\email{gdzhou@math.ecnu.edu.cn}

\begin{document}

\begin{abstract} Let $A$ be a Koszul Calabi-Yau algebra.
We show that  there exists an isomorphism of Batalin-Vilkovisky
algebras between the Hochschild cohomology ring of $A$ and that of its
Koszul dual algebra $A^!$. This confirms (a generalization of) a conjecture of R.~Rouquier.
\end{abstract}

\maketitle

\setcounter{tocdepth}{1}
\tableofcontents


\section{Introduction}

\renewcommand{\thetheorem}{\Alph{theorem}}

The notion of Calabi-Yau algebras is introduced by Ginzburg
(\cite{Ginzburg}), and has been intensively studied in recent years.
Among all Calabi-Yau algebras of great importance and interest are those which are
Koszul. For example, all Koszul Calabi-Yau algebras
are derived from a {\it potential}, which has become a center of the research,
while it is generally difficult to find a potential for non-Koszul Calabi-Yau algebras
(here by being Koszul we mean in the general sense; see Remark \ref{general_Koszul}).
The reader may refer to
\cite{BT,Bocklandt,BSW,Davison,Ginzburg,VdB3,WZ} for more details and examples.

In this paper we show that
for a Koszul Calabi-Yau algebra $A$,
its Hochschild cohomology is isomorphic to the Hochschild cohomology of its Koszul dual algebra $A^!$
as {\it Batalin-Vilkovisky algebras}:
\begin{equation}\label{iso_I}
\mathrm{HH}^{\bullet}(A;A)\simeq\mathrm{HH}^{\bullet}(A^!;A^!).
\end{equation}
This isomorphism has been folklore for years, and some results
are already well-known, such as the Batalin-Vilkovisky algebra structure on both sides.
Let us start with some backgrounds.

Let $A$ be an $n$-Calabi-Yau algebra (see Definition \ref{Def_CY}).
Ginzburg proves that there is a Batalin-Vilkovisky algebra structure
on the Hochschild cohomology of $A$ (\cite[Theorem 3.4.3]{Ginzburg}).
It may be viewed as a noncommutative generalization of the Batalin-Vilkovisky algebra on the
polyvector fields of Calabi-Yau manifolds.
Earlier than that, inspired by Chas-Sullivan's work \cite{CS} on string topology,
Tradler constructs,
for a (differential graded) cyclic associative algebra,
\textit{i.e.} an associative algebra with a non-degenerate cyclically invariant pairing,
a Batalin-Vilkovisky algebra on its Hochschild cohomology (\cite[Theorem 1]{Tradler}).
These two Batalin-Vilkovisky algebra structures are quite different from each other.
For example, consider the space of polynomials of $n$ variables $\mathbb{C}[x_1,x_2,\cdots, x_n]$;
it is a Calabi-Yau algebra in the sense of Ginzburg,
while there does not exist a non-degenerate pairing on it, and therefore Tradler's
construction does not apply.

On the other hand, Ginzburg remarks (\cite[Remark 5.4.10]{Ginzburg}) that
if the Calabi-Yau algebra $A$ is Koszul (in {\it ibid} he also assumes $A$ is of dimension 3, but this
turns out to be not necessary, as shall be shown below), then
its Koszul dual algebra $A^!$ admits a non-degenerate pairing,
and Tradler's construction can be applied to $A^!$.
From the work of
Buchweitz \cite{Buchweitz}, Beilinson-Ginzburg-Soergel \cite{BGS} and Keller \cite{Keller},
we now know that
$
\mathrm{HH}^{\bullet}  (A; A)\simeq\mathrm{HH}^{\bullet} (A^! ;A^!)
$
as Gerstenhaber algebras.
Ginzburg states as a conjecture, which he contributes to Rouquier ({\it ibid} \S5.4), that
this isomorphism is in fact an isomorphism of Batalin-Vilkovisky algebras,
given by Ginzburg and Tradler respectively.
The goal of this paper is to show that this is indeed the case:

\begin{theorem}[Rouquier's conjecture]\label{Main Thm}
Suppose that $A$ is a Koszul Calabi-Yau algebra,
and let $A^!$ be its Koszul dual algebra.
Then there is an isomorphism
$$
\mathrm{HH}^{\bullet}(A ;A)\simeq\mathrm{HH}^{\bullet}(A^{!} ;A^!)
$$
of Batalin-Vilkovisky algebras between the Hochschild cohomology of $A$ and $A^!$.
\end{theorem}

In literature, the two Batalin-Vilkovisky algebra structures on
both sides have been further studied a lot;
see, for example, \cite{Abbaspour,Lambre,Menichi}.
However, relatively less is discussed on their relationships.
The theorem above gives a bridge to connect and unite them.
The key points to prove the above theorem are the following:
\begin{enumerate}
\item[$-$]
For an associative algebra $A$ and its Koszul dual {\it coalgebra} $A^{\ac}$, the
chain complex
$$
(A\otimes A^{\ac}, b),
$$
with the differential $b$ appropriately equipped,
computes the Hochschild homology of $A$ and $A^{\ac}$ simultaneously, that is,
we have canonical isomorphisms
\begin{equation}\label{Can_iso}
\mathrm{HH}_{\bullet} (A)\simeq\mathrm{H}_{\bullet} (A\otimes A^{\ac}, b)\simeq\mathrm{HH}_{\bullet} (A^{\ac}).
\end{equation}

\item[$-$] On both $\mathrm{HH}_{\bullet} (A)$ and $\mathrm{HH}_{\bullet} (A^{\ac})$ the Connes
differential operator
exists;
however, it is in general not easy to define a version of Connes operator on $(A\otimes A^{\ac}, b)$.
Nevertheless, we show that
the canonical isomorphism (\ref{Can_iso})
commutes with the Connes operator on $\mathrm{HH}_{\bullet} (A)$ and $\mathrm{HH}_{\bullet} (A^{\ac})$.

\item[$-$] Analogous to (\ref{Can_iso}), we show that for the Koszul algebra $A$,
there is a canonical complex $(A\otimes A^!, \delta)$ which computes the Hochschild
cohomology of $A$ and $A^!$ simultaneously, where $A^!$ is the Koszul dual {\it algebra}
of $A$, {\it i.e.} there are canonical isomorphisms
\begin{equation}\label{CanIsocohomology}\mathrm{HH}^{\bullet}(A; A)\simeq\mathrm{H}^{\bullet}(A\otimes A^!,\delta)
\simeq\mathrm{HH}^{\bullet}(A^!; A^!).\end{equation}

\item[$-$] By two versions of Poincar\'e duality, due to Van den Bergh and Tradler respectively,
$$
\mathrm{PD}: \mathrm{HH}^{\bullet}(A; A)\simeq\mathrm{HH}_{n-{\bullet}}(A),\quad
\mathrm{HH}^{{\bullet}}(A^!; A^!)\simeq\mathrm{HH}_{n-{\bullet}}(A^{\ac}),
$$
together with the above two results, one obtains the desired isomorphism
$$\mathrm{HH}^{\bullet}(A;A)\simeq\mathrm{HH}^{\bullet}(A^!;A^!),$$
where the Batalin-Vilkovisky operator on each side is the pull-back of the Connes
operator via the Poincar\'e duality.

\item[$-$] The isomorphisms (\ref{Can_iso}) and (\ref{CanIsocohomology}) are not true with respect to the usual grading. They are in fact isomorphism with respect to a certain bigrading; see \S\ref{Sect_HK}  for details.
\end{enumerate}

Quite recently, the above mentioned results of Ginzburg and Tradler are generalised to their twisted versions.
More precisely, when a twisted Calabi-Yau algebra or a Frobenius algebra has semisimple Nakayama automorphism, then its Hochschild cohomology ring is still a Batalin-Vilkovisky algebra; see Kowalzig and
Kr\"ahmer \cite[Theorem 1.7]{KK} and  Lambre, the third author and Zimmermann
\cite[Theorem 0.1]{LZZ14}.
We shall consider the twisted version of our main result and related applications in a future work.

This paper is organized as follows: \S\ref{Section:BarCobar}  recalls some basic facts about bar/cobar constructions and twisting morphisms. Our basic reference of this section is the recent book of Loday
and Vallette (\cite{LV}).
\S\ref{Sect_Hoch} collects the definitions of Hochschild and cyclic (co)homology of algebras and coalgebras;
\S\ref{Sect_Koszul} reviews some basic facts about Koszul algebras;
\S\ref{Sect_HK} computes the Hochschild (co)homology of Koszul algebras and their Koszul dual;
\S\ref{Sect_KCY} studies Koszul Calabi-Yau algebras; \S\ref{Sect_Proof} proves
Theorem \ref{Main Thm};   and the last \S\ref{Sect_App},
gives an application of the previous results to the cyclic homology
of Calabi-Yau algebras.

Throughout this paper, $\k$ denotes a field of arbitrary characteristic.
It is supposed to be of zero characteristic when talking about cyclic (co)homology.

\begin{ack}
The authors  would like to thank Farkhod Eshmatov,
Ji-Wei He and Dong Yang for many helpful conversations.
The first two authors are partially supported by  NSFC (No. 11271269).
The third author is partially supported by Shanghai Pujiang
Program (No. 13PJ1402800),  by NSFC
(No. 11301186) and by the Doctoral Fund of Youth Scholars of Ministry of Education of China
(No. 20130076120001).

Nearly at the same time when the first version of paper was put in arXiv,
a paper of Estanislao Herscovich also appeared  in arXiv (\cite{Herscovich}), which also
deals with Koszul duality and  Hochschild (co)homology.
There is certain overlap between our paper and his,
mostly in Sections 2-5. We  recommend his paper as a companion to ours. Estanislao proves that the Hochschild (co)homology of a Koszul algebra and that of its Koszul dual form dual differential calculi in a certain sense, 
 he, however,  does not consider Calabi-Yau algebras and Batalin-Vilkovisky structures.
\end{ack}


\setcounter{theorem}{0}
\renewcommand{\thetheorem}{\arabic{theorem}}

\section{Bar/cobar construction and twisting morphisms}\label{Section:BarCobar}

The goal of  this section is to   recall the bar/cobar construction and twisting morphisms;
for details, we refer the reader to Loday-Vallette \cite{LV}.

 In this paper, we shall use  chain complexes and homological grading everywhere,
 that is, a complex is of the form
$$
\xymatrixcolsep{3.5pc}
\xymatrix{\cdots \ar[r]^{d_{n+2}}&C_{n+1}\ar[r]^{d_{n+1}}&
C_n\ar[r]^{d_n}&C_{n-1}\ar[r]^{d_{n-1}}&\cdots
}$$
and the differential decreases the index.

Let $V$ be a $\k$-vector space, denote $V^*=\mathrm{Hom}_\k(V, \k)$ be its $\k$-dual.
For an abelian group $G$, written additively, a $G$-graded vector space $V=\oplus_{g\in G}V_g$
is a direct sum of vector spaces indexed by $G$. For a $G$-graded $\k$-vector space
$V=\oplus_{g\in G} V_g$, an element of $V_g$ is said to be homogenous of complete
degree $g$.  This $G$-graded $\k$-vector space $V=\oplus_{g\in G} V_g$ is locally finite
if for any $g\in G$, $\mathrm{dim}(V_g)<\infty$.  Denote by $V^\vee$ its graded dual, that is,
for $g\in G$, $(V^\vee)_g=(V_{-g})^*$. A linear map $f: V\to W$ between $G$-graded
vector spaces is said to be  graded  of complete degree $g\in G$, if for any $g'\in G$,
$f(V_{g'})\subseteq W_{g'+g}$. For two $G$-graded vector spaces $V$ and $W$,
their tensor product $V\ot W$ and the space of all graded linear maps of arbitrary complete degrees
$\mathrm{Hom}(V, W)$ are also $G$-graded.
We are mainly concentrated in the case where $G=\mathbb{Z}$ or $\mathbb{Z}\times \mathbb{Z}$.
When $G=\mathbb{Z}$, $G$-graded spaces are simply called graded and  the complete degree is
also called degree; when   $G=\mathbb{Z}\times \mathbb{Z}$, $G$-graded spaces are called
bigraded and the complete degree is denoted by (degree, weight), that is, we call the first
index ``degree" and the second ``weight". For $v\in V_i$ for
$G=\mathbb{Z}$ or $v\in V_{ij}$ for $G=\mathbb{Z}\times \mathbb{Z}$,
we write $\mathrm{deg}(v)=|v|=i$ and $\mathrm{weight}(v)=j$.
Given two $G$-graded vector spaces $V$ and $W$,  there exists a natural map
$V^\vee\ot W^\vee\to (V\ot W)^\vee$ sending $\varphi\ot \psi$ to the map
$v\ot w\mapsto (-1)^{\mathrm{deg}(\psi)\mathrm{deg}(v)} \varphi(v)\psi(w)$.
If $V$ and $W$ are locally finite,
then the above map is an isomorphism.  There is also a natural injection
 \begin{equation}\label{TransferingHomToTensorProduct}
 W\ot V^\vee\to \mathrm{Hom}(V, W): w\ot \varphi\mapsto (v\mapsto \varphi(v)w).
 \end{equation}
It is an isomorphism if $V$ is locally finite.

Given a bigraded vector space $V=V_{{\bullet}{\bullet}}$, denote by
$sV$ its degree shift, that is, $(sV)_{ij}=V_{i-1, j}$. We shall also use
the complete shift $V(\ell)$ such that $ V(\ell)_{ij}=V_{i-\ell, j-\ell}$.
Sometimes we  consider bigraded maps of complete degree $(r, 0)\in \mathbb{Z}\times \mathbb{Z}$,
that is, bigraded maps which preserve the weight; in this case, we shall say that these are graded
maps of degree $r$. Differentials are graded maps of degree $-1$ and square to zero, that is,
they preserve the weight and decrease the degree by $1$. In the following,
we sometimes write $\mathrm{deg}(f)=|f|$.

 A $\mathbb{Z}$-graded vector space $V$ is connected if $V_i=0$ for $i<0$ and $V_0=\k$,
 and it is $2$-connected if it is connected and $V_1=0$; a bigraded vector space is
 (Adams-)connected if $V_{ij}=0$ when $i\leq 0$ or $j\leq 0$ except $V_{00}=\k$.

 \indexspace

An associative algebra is a triple $(A,  \mu, \eta)$,  where $A$ is a vector space,
$\mu: A\otimes A\to A$ is the multiplication and $\eta: \k\to A$ is the unit,
which satisfy the obvious conditions. An algebra homomorphism is a linear map
between two algebras which commute with multiplications and units.
An augmented algebra is an algebra $(A, \mu, \eta)$ endowed with an
algebra homomorphism $\epsilon:   A\to \k$.
For an augmented algebra $A$, we have $A=\k 1_A \oplus \ol{A}$, where
$1_A=\eta(1)$ (thus $\k 1_A=\mathrm{Im}(\eta)$) and $\ol{A}=\mathrm{Ker}(\epsilon)$.
We can also define (bi-)graded algebras, augmented (bi-)graded algebras,
differential (bi-)graded algebras, augmented differential (bi-)graded algebras.
For these notions, multiplications, units and algebra maps are assumed to preserve
the degree (and also the weight in the bigrading setup), differentials
decrease the degree by $1$ (and preserve the weight in the bigrading setup).
For an (augmented diffential bigraded) algebra $A$,  $A^{op}$ denotes its opposite algebra,
which has the same underlying space (and the same unit, the same differential, the same
grading and the same augmentation) as $A$, but is
endowed with the new product $a*b=(-1)^{|a||b|}ba$.

A coassociative coalgebra is a triple $(C,  \Delta, \epsilon)$ where $C$ is a vector space,
$\Delta: C\to C\ot C$ is the comultiplication and $\epsilon:   C\to \k$ is the counit,
which satisfy the obvious conditions. A coalgebra homomorphism is a linear map
between two coalgebras which commute with comultiplications and counits.
A coaugmented coalgebra is a coalgebra $(C, \Delta, \epsilon)$ endowed with a coalgebra
homomorphism $\eta: \k \to C$. For a coaugmented coalgebra $C$, we have $C=\k 1_C\oplus \ol{C}$,
where $1_C=\eta(1)$ (thus $\k 1_C=\mathrm{Im}(\eta)$) and $\ol{C}=\mathrm{Ker}(\epsilon)$.
 We can also define (bi-)graded coalgebras, augmented (bi-)graded coalgebras, differential
 (bi-)graded coalgebras, coaugmented differential (bi-)graded coalgebras. For these notions,
 comultiplications, counits and coalgebra maps are assumed to preserve  the degree
 (and the weight in bigraded setup).  For a coaugmented coalgebra $C$,
 let $\overline\Delta^0=id:  C\to C$, $\overline\Delta^1=\overline\Delta: C\to C\ot C$ be $\ol{\Delta}(x)=\Delta(x)-x\ot 1-1\ot x$ and for
$n\geq 2$, $\ol\Delta^{n}=(\ol \Delta\ot \mathrm{Id}\ot \cdots \ot \mathrm{Id})\circ \ol\Delta^{n-1}: C\to C^{\ot n}$.
A coalgebra is conilpotent if $C=\cup_{i\geq 0} \mathrm{Ker}(\ol\Delta^i)$.
For a (coaugmented differential bigraded) coalgebra $C$,  $C^{coop}$
denoted its opposite coalgebra, which has the same underlying space
(and the same counit, the same differential, the same grading and the same coaugmentation)
as $C$, but is endowed with the new coproduct $\Delta_{C^{coop}}=\tau\circ \Delta_C$,
where $\tau: C\ot C\to C\ot C$ is the swap sending $c_1\ot c_2\mapsto (-1)^{|c_1||c_2|}c_2\ot c_1$.

Given a coalgebra $(C, \Delta, \epsilon_C)$ and an algebra $(A, \mu, \eta_A)$,
the linear space $\mathrm{Hom}(C, A)$ is an algebra with respect to the convolution product,
that is, for $f, g\in \mathrm{Hom}(V, W)$, $f*g=\mu\circ (f\ot g)\circ \Delta$.
The unit of this algebra is $\eta_A\circ \epsilon_C$.
When $C$ and $A$ are differential (bi-)graded, $\mathrm{Hom}(C, A)$ denotes the space of
(bi-)graded linear maps. In this case, $\mathrm{Hom}(C, A)$ is differential graded algebra
whose differential $\partial$ is induced from that of $A$ and $C$,
that is $\partial(f)=d_A\circ f-(-1)^{|f|}f\circ d_C$, where $d_A$ (resp. $d_C$)
is the differential over $A$ (resp. $C$).

For a differential (bi-)graded coalgebra $C$ and a differential (bi-)graded algebra $A$,
an element $\alpha\in \mathrm{Hom}(C, A)$ of degree $-1$ (or complete degree $(-1, 0)$
in the bigrading setup)  is said to be a twisting morphism if $\partial(\alpha)+\alpha*\alpha=0$
(\cite[Section 2.1.3]{LV}).  When $C$ is coaugmented  and $A$ is  augmented, a twisting
morphism is always assumed to satisfy $\alpha\circ \eta_C=0$ and $\epsilon_A\circ \alpha=0$.
Write $\mathrm{Tw}(C, A)$ for the space  of twisting morphisms.

For a differential (bi-)graded coalgebra $C$ and a differential (bi-)graded algebra $A$,
given a twisting morphism $\alpha$ from $C$ to $A$, one can define a new differential
over $\mathrm{Hom}(C, A)$, namely,  for $f\in \mathrm{Hom}(C, A)$,
define $\partial_\alpha(f)=\partial(f)+[\alpha, f]$, where $[\alpha, f]=\alpha*f-(-1)^{|\alpha|}f*\alpha$.
The triple $(\mathrm{Hom}(C, A), *, \partial_\alpha)$ is a differential (bi-)graded algebra, denoted by
$\mathrm{Hom}^\alpha(C, A)$, called the twisted convolution algebra (\cite[Proposition 2.1.6]{LV}).
When $C$ is locally finite, there exists an isomorphism of (bi-)graded vector spaces
$\mathrm{Hom}(C, A)\simeq A\ot C^\vee$,   and in this case, there is an induced
differential (bi-)graded algebra structure over
$A\ot  C^\vee$ and we denote it by $A\otimes^\alpha C^\vee$ \footnote{This notation seems to be new,
which does not appear in Loday-Vallette \cite{LV}.}.

Given  a differential (bi-)graded coalgebra $C$ and a differential (bi-)graded algebra $A$,
for an element $\alpha\in \mathrm{Hom}(C, A)$ of degree $-1$ (or complete degree $(-1, 0)$
in the bigrading setup), one can also define a new differential on the tensor product $C\ot A$ by
posing $d_\alpha=d_C\ot \mathrm{Id}_A+\mathrm{Id}_C\ot d_A+d^r_\alpha$, where
$d^r_\alpha=(\mathrm{Id}_C\ot \mu)\circ (\mathrm{Id}_C\ot
\alpha\ot \mathrm{Id}_A)\circ (\Delta\ot \mathrm{Id}_A)$.
This is a differential if (and only if) $\alpha$ is a twisting morphism (\cite[Lemma 2.1.5]{LV}).
 The tensor product $C\ot A$ endowed with $d_\alpha$ is denoted by
 $C\ot_\alpha A$, called the (right) twisted tensor product. One can also define
 left twisted tensor product $A\ot_\alpha C$ by defining the differential to be
 $d_C\ot \mathrm{Id}_A+\mathrm{Id}_C\ot d_A+d^l_\alpha$ where
 $d^l_\alpha=(\mu\ot\mathrm{Id}_C)\circ (\mathrm{Id}_A\ot
 \alpha\ot \mathrm{Id}_C)\circ (\mathrm{Id}_A\ot\Delta)$.

Let $\alpha: C\to A$ and $\alpha': C'\to A'$ be two twisting morphisms
from (bi-)graded coalgebras to (bi-)graded algebras.
Let $f: C\to C'$ (resp. $g: A\to A'$) be a homomorphism of
(bi-)graded coalgebras (resp. of (bi-)graded algebras).
We say that the two twisting morphisms are compatible if $g\circ \alpha= \alpha'\circ f$.
The Comparison Lemma for twisted tensor products (\cite[Lemma 2.1.9]{LV}) says that
if $f$ and $g$ are quasi-isomorphisms, then so is
$f\ot g: C\ot_\alpha A\to C'\ot_{\alpha'} A'$.

\bigskip

Let $A=\k 1_A\oplus \ol{A}$ be an augmented (bi-)graded algebra.
Its bar construction $\mathrm{B}(A)$ is defined to be the tensor
 coalgebra $T^c(s\ol{A})=\k\oplus s\ol{A}\oplus (s\ol{A})^{\ot 2}\oplus \cdots$.
We shall write an element of $\mathrm{B}(A)$ by
$[a_1|a_2|\cdots |a_n]\in (s\ol{A})^{\ot n}$ for $a_1, \cdots, a_n\in \ol{A}$.
We can put a bigrading on the bar construction $\mathrm{B}(A)$.
When $A$ is bigraded, $B(A)$ is naturally bigraded. When $A$ is not bigraded,
we give $A$ a bigrading as follows:  if $A$ is only an augmented ungraded algebra,
we place $A$ in degree zero,  and put  $\k 1_A$ in weight zero and $\ol{A}$ in weight  $1$;
if $A$ is  an augmented (differential) graded  algebra, we  place its original grading in the degree
grading and put  $\k 1_A$ in weight zero and $\ol{A}$ in weight  $1$.
As $T^c(s\ol{A})$ is a cofree conilpotent coalgebra, the map
$T^c(S\ol{A})\twoheadrightarrow s\ol{A}\ot s\ol{A}  \stackrel{f}{\to} s\ol{A}$  induces a coderivation on
$\mathrm{B}(A)$,  denoted by $d_2$. Here $f: s\ol{A}\ot s\ol{A}\to s\ol{A}$
sends $[a]\ot [b]$ to $(-1)^{\mathrm{deg}(a)}[ab]$. The differential of $A$ also induces
a differential $d_1$ on $\mathrm{B}(A)$. Now $(\mathrm{B}(A), d=d_1+d_2)$
is a differential bigraded conilpotent cofree coalgebra. The canonical map
$\pi: \mathrm{B}(A)\twoheadrightarrow s\ol{A} \simeq \ol{A} \rightarrowtail A$
is a twisting morphism and the right twisted tensor product
$\mathrm{B}(A)\ot_{\pi}A$ (resp. the left twisted tensor product $A\ot_{\pi}\mathrm{B}(A)$)
is quasi-isomorphic to $\k$ and thus is a free resolution of $\k$ as right module (resp. left module);
see \cite[Proposition 2.2.8]{LV}.
For a quasi-isomorphism of augmented  differential graded connected algebras $g: A\to A'$,
the induced map $\mathrm{B}(g): \mathrm{B}(A)\to \mathrm{B}(A')$ is also a quasi-isomorphism
(\cite[Proposition 2.2.3]{LV}).  The proof can be done with the help of a spectral sequence induced by
the weight grading on the bar construction, that is, $F_p(\mathrm{B}(A))=\oplus_{i\leq p} (s\ol{A})^{\ot i}$.

For a coaugmented (bi-)graded coalgebra $C$, its cobar construction $\Omega(C)$  is a
differential bigraded algebra. Let $C=\k 1_C\oplus \ol{C}$ be a coaugmented (bi-)graded coalgebra.
Its cobar construction $\Omega(C)$ is defined to be the tensor algebra $T(s^{-1}\ol{C})$.
We shall write an element of $\Omega(A)$ by $\langle a_1|a_2|\cdots |a_n\rangle \in (s^{-1}\ol{A})^{\ot n}$.
The bigrading on $\Omega(C)$ can be given similarly.
When $C$ is bigraded, $\Omega(C)$ is naturally bigraded. When $C$ is not bigraded,
we give $C$ a bigrading as follows:  if $C$ is only an augmented ungraded coalgebra,
we place $C$ in degree zero,  and put  $\k 1_C$ in weight zero and $\ol{C}$ in weight  $1$;
if $C$ is  an augmented (differential) graded  coalgebra, we  place its original grading in the
degree  grading and put  $\k 1_C$ in weight zero and $\ol{C}$ in weight  $1$.
As $T(s^{-1}\ol{C})$ is a free algebra, the map
$s^{-1}\ol{C} \stackrel{f}{\to} s^{-1}\ol{C}\ot s^{-1}\ol{C} \rightarrowtail T(s^{-1}\ol{C})$
induces a derivation $d_2$ on $\Omega(C)$,  where $f: s^{-1}\ol{C} \to s^{-1}\ol{C}\ot s^{-1}\ol{C}$
sends $\langle c \rangle $ to $(-1)^{|c_{(1)}|}\langle c_{(1)}\rangle \ot \langle c_{(2)}\rangle$.
The differential of $C$ also induces a differential $d_1$ on $\Omega(C)$.
Now $(\Omega(C), d=d_1+d_2)$ is a differential bigraded  free algebra. The map
$\iota: C\twoheadrightarrow \ol{C}\simeq s^{-1}\ol{C}\rightarrowtail \Omega (C)$
is a twisting morphism and the twisted tensor product
$C\ot_{\iota}\Omega(C)$ (and $\Omega(C)\ot_{\iota} C$) is quasi-isomorphic to $\k$;
see \cite[Proposition 2.2.8]{LV}.
For a quasi-isomorphism of augmented  differential graded $2$-connected
coalgebras $h: C\to C'$, the map $\Omega(h): \Omega(C)\to \Omega(C')$ is
also a quasi-isomorphism (\cite[Porposition 2.2.5]{LV}).
The proof can be done with the help of a spectral sequence induced by
the weight grading on the cobar construction, that is,
$F_p(\Omega(C))=\oplus_{i\geq -p} (s^{-1}\ol{C})^{\ot i}$.

The bar construction and the cobar construction form a pair of adjoint functors and represent the bifunctor
$\mathrm{Tw}(C, A)$.  In fact, given an augmented differential (Adams-)connected (bi-)graded algebra $A$
and a coaugmented differential (Adams-)connected conilpotent  (bi-)graded coalgebra $C$,
there exist bifunctorial isomorphisms
$$
\mathrm{Hom}(C, \mathrm{B}(A))\simeq \mathrm{Tw}(C, A)\simeq \mathrm{Hom}(\Omega(C), A),
$$
where the first $\mathrm{Hom}$ is in the category of augmented differential (bi-)graded coalgebras and
the second in the category of coaugmented differential connected conilpotent  (bi-)graded algebras.
Applying this adjunction to $C=\mathrm{B}(A)$, we obtain the counit
$\epsilon: \Omega\mathrm{B}(A)\to A$ and the universal twisting morphism $\pi: \mathrm{B}(A)\to A$
appeared above; similarly, we obtain the unit $\eta: C\to \mathrm{B}\Omega (C)$ and the universal
twisting morphism $\iota: C\to \Omega(C)$ mentioned before.
Given a twisting morphism $\alpha: C\to A$, it factors uniquely through $\pi$ and $\iota$:
$$\xymatrix{ & \Omega(C) \ar@{..>}[dr]^{g_\alpha}& \\
C\ar[ur]^{\iota} \ar[rr]^\alpha \ar@{..>}[dr]_{f_\alpha}& & A\\ & \mathrm{B}(A)\ar[ur]_\pi& }$$
where $f_\alpha$ is a homomorphism of augmented differential (bi-)graded algebras and $g_\alpha$ a homomorphism between  coaugmented differential connected conilpotent  (bi-)graded coalgebras.

A \textit{fundamental result about twisting morphisms} (\cite[Theorem 2.3.1]{LV}
is the following: Let $C$ (resp. $A$) be a connected (bi-)graded differential coalgebra
(resp. algebra) and let $\alpha: C\to A$ be a twisting morphism.
Then $C\ot_\alpha A$ is a resolution of $\k$ if and only if the induced
homomorphism of differential (bi-)graded algebras $g_\alpha: \Omega(C)\to A$
corresponding to $\alpha$ is an quasi-isomorphism if and only if the induced
homomorphism of differential (bi-)graded coalgebras
$f_\alpha: C\to \mathrm{B}(A)$ corresponding to $\alpha$ is an quasi-isomorphism.
If this is the case, we call such a twisting morphism \textit{Koszul morphism}.

Given a twisting morphism $\alpha: C\to A$, one can define a twisted differential on
$A\ot C\ot A$; see \cite[Exercise 2.7.6]{LV}. In fact,
let $d_\alpha=d_{A\ot C\ot A}+\mathrm{Id}_A\ot d^r_\alpha-d^l_\alpha\ot \mathrm{Id}_A$.
Then it is a differential and write $A\ot_\alpha C\ot_\alpha A$ for this new complex.
There is an isomorphism of complexes
\begin{equation}\label{BimoduleResolutionComputingtwisteHom}
\mathrm{Hom}^\alpha(C, A)\simeq \mathrm{Hom}_{A^e}(A\ot_\alpha C\ot_\alpha A, A).
\end{equation}
Furthermore the bimodule homomorphism
$A\ot_\alpha C\ot_\alpha A\twoheadrightarrow A\ot \k \ot A\twoheadrightarrow A$
is an quasi-isomorphism of bimodules if and only if $\alpha$ is a Koszul morphism.
So when $\alpha$ is a Koszul morphism, $A\ot_\alpha C\ot_\alpha A$ is bimodule free resolution of $A$.
Let $\alpha': C'\to A$ be another twisting morphism. Suppose that   $\alpha$ and $\alpha'$ are compatible
via $f: C\to C'$. Then there is an induced morphism of compexes
$A\ot_\alpha C\ot_\alpha A\to A\ot_{\alpha'} C'\ot_{\alpha'} A$ and this induces
  $\mathrm{Hom}^{\alpha'}(C', A)\to \mathrm{Hom}^\alpha(C, A)$.
When $C$ and $C'$ are locally finite, this in turn induces
\begin{equation}\label{InducedMorphismBetweenNewTwistedTensorProduct}
A\ot^{\alpha'}{C'}^\vee\to A\ot^{\alpha}C^\vee.\end{equation}

Let $C$ be  an (Adams-)connected differential (bi-)graded coalgebra
$C$ and let $A$ be an (Adams-) connected differential (bi-)graded algebra $A$.
Suppose that they are locally finite, so $C^\vee$ is an algebra and $A^\vee$ becomes a coalgebra.
Given a twisting morphism $\alpha: C\to A$, then its dual
$\alpha^\vee: A^\vee\to C^\vee$ is also a twisting morphism. It is easy to show that the natural map
$$\mathrm{Hom}^\alpha(C, A)\to \mathrm{Hom}^{\alpha^\vee}
(A^\vee, C^\vee), \varphi\mapsto (f\mapsto (-1)^{|f||\varphi|} f\circ \varphi)$$
is an quasi-isomorphism of differential (bi-)graded algebras.
When $C$ and $A$ are locally finite, this translates into an isomorphism
\begin{equation}\label{DualTwistingMorphismsInducingIsomorphism}
A\ot^\alpha C^{\vee}\simeq  C^\vee\ot^{\alpha^\vee} A.
\end{equation}

\section{Hochschild homology of algebras and coalgebras}\label{Sect_Hoch}

This section briefly recalls the definitions of Hochschild and cyclic homology
of algebras and coalgebras.
The materials are well-known, however, it is rare to find a reference which
treats algebras and coalgebras simultaneously.
They will be used in later sections.
We shall always deal with (co)augmented ordinary (co)algebras in order to avoid
complicated expressions of signs. The reader may add the signs himself using the Koszul sign rule.

\subsection{Hochschild homology of algebras}

\begin{definition}[Hochschild homology]
Suppose that $A$ is an algebra. The {\it Hochschild chain complex} of $A$
is the graded vector space
$$\mathrm{CH}_{\bullet}(A):=\bigoplus_{n\ge 0}A\otimes A^{\otimes n}$$
with differential $b:\mathrm{CH}_{\bullet}(A)\to\mathrm{CH}_{{\bullet}-1}(A)$ defined by
\begin{equation}\label{Hochbcpx}
b(a_0, a_1,\cdots, a_n):=\sum_{i=0}^{n-1}(-1)^i(a_0, \cdots, a_{i-1}, a_ia_{i+1},\cdots, a_n)+(-1)^{n}(a_na_0, a_1,\cdots, a_{n-1}).
\end{equation}
The associated homology is called the {\it Hochschild homology} of $A$ and is denoted
by $\mathrm{HH}_\bullet(A)$.
\end{definition}


As we have assumed that $A$ is augmented and unital, let $\overline A$ be the augmentation ideal, then
$\mathrm{CH}_{\bullet}(A)$ is quasi-isomorphic to the {\it normalised  Hochschild complex}
$$\overline{\mathrm{CH}}_{\bullet}(A):=A\ot \mathrm{B}(A)=\bigoplus_{n\geq 0} A\ot \ol{A}^{\ot n}$$
with the induced differential $b$ (see, for example, Loday \cite[Proposition 1.6.5]{Loday}).
It is easy to see that there is an isomorphism of complexes
$$\overline{\mathrm{CH}}_{\bullet}(A)\simeq A\ot_{A^e} (A\ot_\pi \mathrm{B}(A)\ot_\pi A).$$

From now on, when mentioning the normalised  Hochschild complex,
we take this identification.

\begin{definition}[Connes cyclic operator]
Let $A$ be an (augmented differential graded)  algebra.
The {\it Connes cyclic operator}
$$B: \overline{\mathrm{CH}}_{\bullet}(A)\to\overline{\mathrm{CH}}_{{\bullet}+1} (A)$$
is defined by
$$
B(a_0, a_1,\cdots, a_n)
:=\sum_{i=0}^{n} (-1)^{(i-1)(n+1)} (1,   a_i, \cdots,   a_n, a_0,\cdots, a_{i-1}),
$$
where, by abuse of notations, we identify an element of $A$ with its  image
under the natural map $A\to\overline A$.
\end{definition}

It is easy to check that $B^2=0$ and $Bb+bB=0$, and hence
$$(\overline{\mathrm{CH}}_{\bullet}(A), b, B)$$
defines a {\it mixed complex}, in the sense of Kassel \cite{Kassel}.

\begin{definition}[Cyclic homology; Jones \cite{Jones}]\label{def_HC}
Let $A$ be an algebra.
Let $\ol{\mathrm{CH}}_{\bullet}(A)$ be the normalised  Hochschild complex of $A$ (in fact, this
definition applies to any mixed complex),
and
$u$ be a free variable of degree $-2$, which commutes with $b$ and $B$.
The (reduced) {\it negative cyclic, periodic cyclic, }
and {\it cyclic} chain complex of $A$ are the following complexes
\begin{eqnarray*}
(\overline{\mathrm{CH}}_{\bullet}(A)[[u]], b+uB),&& \\
(\overline{\mathrm{CH}}_{\bullet}(A)[[u, u^{-1}]], b+uB),&& \\
(\overline{\mathrm{CH}}_{\bullet}(A)[[u, u^{-1}]]/u\overline{\mathrm{CH}}_{\bullet}(A)[[u]], b+uB),
\end{eqnarray*}
and are denoted by $\mathrm{CC}^{-}_{\bullet}(A),\mathrm{CC}^{\mathrm{per}}_{\bullet}(A)$
and $\mathrm{CC}_{\bullet}(A)$ respectively.
The associated homology are called the {\it negative cyclic, periodic cyclic}
and {\it cyclic homology} of $A$, and are denoted by
$\mathrm{HC}_{\bullet}^{-}(A)$, $\mathrm{HC}_{\bullet}^{\mathrm{per}}(A)$ and
$\mathrm{HC}_{\bullet}(A)$ respectively.
\end{definition}

From the definition, we see that there is a short exact sequence
$$0\longrightarrow
\ol{\mathrm{CH}}_{\bullet}(A)
\longrightarrow
\mathrm{CC}_{\bullet}(A)
\stackrel{\cdot u}\longrightarrow
\mathrm{CC}_{{\bullet}-2}(A)
\longrightarrow
0,$$
which induces on the homology level the so-called {\it Connes exact sequence}:
\begin{equation}
\cdots\longrightarrow
\mathrm{HH}_{\bullet}(A)
\longrightarrow
\mathrm{HC}_{\bullet}(A)
\longrightarrow
\mathrm{HC}_{{\bullet}-2}(A)
\longrightarrow
\mathrm{HH}_{{\bullet}-1}(A)
\longrightarrow
\cdots.
\end{equation}

\begin{definition}[Hochschild cohomology]\label{Def_HochCo}
Let $A$ be an associative algebra, and $M$ be an $A$-bimodule.
The {\it Hochschild cochain complex} $\mathrm{CH}^{\bullet} (A; M)$ of $A$ with value in $M$
is the complex whose underlying space is
$$
\bigoplus_{n\ge 0}\mathrm{Hom}(A^{\otimes n}, M)
$$
with coboundary $\delta:\mathrm{Hom}(A^{\otimes n}, M)\to \mathrm{Hom}(A^{\otimes n+1}, M)$ defined by
\begin{eqnarray}\label{Hoch_diff}
 (\delta f)(a_0, a_1,a_2,\cdots,a_n) &=& a_0 f(a_1,\cdots,a_n)+
 \sum_{i=0}^{n-1}(-1)^{i+1} f(a_0,\cdots, a_ia_{i+1},\cdots,a_n) \nonumber\\
    && +(-1)^{n}f(a_0,\cdots,a_{n-1})a_n.
\end{eqnarray}
The associated cohomology is called the {\it Hochschild cohomology} of $A$ with value in $M$,
and is denoted by $\mathrm{HH}^{\bullet}(A; M)$.
In particular, if $M=A$,
then $\mathrm{HH}^{\bullet}(A; A)$  is called the {\it Hochschild cohomology}
of $A$.
\end{definition}

\begin{definition}
Let $A$ be an associative algebra and let $\mathrm{CH}^{\bullet}(A;A)$
be its Hochschild cochain complex.
\begin{enumerate}
\item[$(1)$]
The {\it Gerstenhaber cup product} on
$\mathrm{CH}^{\bullet}(A;A)$
is defined as follows: for any $f\in \mathrm{CH}^{n}(A;A)$ and $g \in \mathrm{CH}^{m}(A;A)$,
$$
f\cup g(a_1,\ldots,a_{n+m}):=(-1)^{nm}f(a_1,\ldots,a_{n})g(a_{n+1},\ldots,a_{n+m}).
$$

\item[$(2)$]
For any $f\in \mathrm{CH}^{n}(A;A)$ and $g \in \mathrm{CH}^{m}(A;A)$,
let
\begin{eqnarray*}
&&
f\circ g (a_1,\ldots,a_{n+m-1})\\
&:=&\sum^{n-1}_{i=1}(-1)^{(m-1)(i-1)}f(a_1,\ldots,a_{i-1},g(a_{i},\ldots,a_{i+m-1}),a_{i+m},\ldots,a_{n+m-1}).
\end{eqnarray*}

The {\it Gerstenhaber bracket} on $\mathrm{CH}^{\bullet}(A;A)$
is defined to be
$$
\{f,g\}:=f\circ g-(-1)^{(n-1)(m-1)}g\circ f.
$$
\end{enumerate}
\end{definition}

Both the Gerstenhaber product and the Gerstenhaber
bracket induce a well-defined product and bracket on Hochschild cohomology
$\mathrm{HH}^{\bullet}(A;A)$, which makes $\mathrm{HH}^{\bullet}(A;A)$
into a Gerstenhaber algebra.
Recall that a {\it Gerstenhaber algebra} is
a graded commutative algebra together with a degree $-1$
Lie bracket $\{-,-\}$ such that
$$
a\mapsto\{a,b\}
$$
are derivations with respect to the product.

\begin{theorem}[Gerstenhaber]\label{G-algebra}
Let $A$ be an algebra.
Then the Hochschild cohomology $\mathrm{HH}^{{\bullet}}(A;A)$ of $A$ equipped with the
Gerstenhaber cup product and the Gerstenhaber bracket forms a Gerstenhaber algebra.
\end{theorem}

\begin{proof}
For a proof, see Gerstenhaber \cite[Theorems 3-5]{Gerstenhaber}.
\end{proof}

Similarly to the Hochschild homology case, one may
introduce the reduced Hochschild cochain complex $\overline{\mathrm{CH}}^{\bullet}(A;A)$
(see Loday \cite[\S1.5.7]{Loday}).
It turns out that
$$\overline{\mathrm{CH}}^{\bullet}(A;A)
\simeq\mathrm{Hom}^\pi(\mathrm B(A), A),$$
where  the RHS is the twisted convolution algebra induced by the canonical Koszul morphism
$\pi: \mathrm{B}(A)\to A$.


%



\begin{proposition}\label{Interpret_HH}
Let $A$ be an algebra.
Denote by $A^e=A\otimes A^{\mathrm{op}}$ the enveloping algebra of $A$.
View $A$ as an $A^e$-module, then
$$
\mathrm{HH}_{\bullet} (A)\simeq\mathrm{Tor}^{A^e}_{\bullet} (A,A),\quad
\mathrm{HH}^{\bullet}(A)\simeq\mathrm{Ext}^{\bullet}_{A^e}(A,A).
$$
\end{proposition}

\begin{proof}
See, for example, Weibel \cite[Lemma 9.1.3]{Weibel}.
\end{proof}

The {\it cyclic cochain complex} of an associative algebra $A$
is defined to be dual complex of $\mathrm{CC}_{\bullet}(A)$, with
the dual differential $b^*$ and $B^*$ respectively.
Let $v$ be the dual variable of $u$, which is of degree 2. Then $\mathrm{CC}^{\bullet}(A)$
is a module over $\k[[v]]$.

There is no short exact sequence which relates the Hochschild cochain complex
$\mathrm{CH}^{\bullet}(A;A)$
and the cyclic cochain complex $\mathrm{CC}^{\bullet}(A)$;
instead, we consider the dual complex $\mathrm{Hom}(\mathrm{CH}_{\bullet}(A), \k)$ of
$\mathrm{CH}_{\bullet}(A)$  and then
$$
\xymatrixcolsep{3.5pc}
\xymatrix{0\ar[r]&
v\cdot\mathrm{CC}^{{\bullet}-2}(A)\ar[r]^-{\mathrm{embedding}}&
\mathrm{CC}^{{\bullet}}(A)\ar[r]^-{\mathrm{projection}}&
\mathrm{Hom}(\mathrm{CH}_{\bullet}(A), \k)\ar[r]&
0}$$
is exact,
which induces the Connes long exact sequence on the cohomology level
$$
\cdots\longrightarrow
\mathrm{Hom}(\mathrm{HH}_{\bullet}(A), \k)\longrightarrow
\mathrm{HC}^{{\bullet}-1}(A)\longrightarrow
\mathrm{HC}^{{\bullet}+1}(A)\longrightarrow
\mathrm{Hom}(\mathrm{HH}_{\bullet}(A), \k)\longrightarrow\cdots,
$$
where the isomorphism $\mathrm{H}^{\bullet}(v\cdot\mathrm{CC}^{\bullet}(A))
\simeq \mathrm{HC}^{{\bullet}-2}(A)$
is used, due to the isomorphism of chain complexes
$v\cdot \mathrm{CC}^{\bullet} (A)\stackrel{/v}\simeq\mathrm{CC}^{{\bullet}}(A)$.

\subsection{Hochschild homology of coalgebras}

The Hochschild homology of {\it coalgebras}
arises from Algebraic topology as examples of cosimplicial objects ({\it cf.} Eilenberg-Moore \cite{EM}).

\begin{definition}[Hochschild homology of coalgebras]\label{Hoch_coalg}
Suppose that $C$ is a coaugmented coalgebra with a counit and co-augmentation such that
$C=\k\oplus \overline C$.
Write the coproduct by $\Delta(c)=\sum_{(c)} c'\otimes c''+c\ot 1+1\ot c$, for any $c\in \ol{C}$,
then the {\it reduced Hochschild complex} is
$$
\overline{\mathrm{CH}}_{\bullet}(C):=\Omega(C)\ot C=\bigoplus_{n\ge 0} \overline C^{\otimes n}\otimes C
$$
with the Hochschild differential $b$ and the Connes cyclic operator $B$ defined by
\begin{eqnarray}
b( a_1,\cdots,  a_n,a_0)&:=&
\sum_{i=1}^{n}\sum_{(a_i)}
(-1)^{i}(a_1,\cdots, a_{i-1},\overline a_i',\overline a_i'',\cdots, a_n, a_0)\nonumber\\
&&+(-1)^{n+1}\cdot\sum_{(a_0)}((a_1,\cdots, a_n,\overline a_0', a_0'')
-(\overline a_0'', a_1,\cdots, a_{n}, a_0')),\label{Hochbcopx_coalg}\\
B(a_1,\cdots, a_n, a_0)&:=&
\sum _i(-1)^{i(n-i)}\varepsilon(a_0)(a_{i+1},\cdots, a_n, a_1,\cdots,
a_{i-1}, a_i),\nonumber
\end{eqnarray}
respectively, where $\overline c$ is again the image of $c\in C$ in $\overline C$,
and $\varepsilon: C\to\k$ is the counit.
The homology $(\overline{\mathrm{CH}}_{\bullet} (C), b)$ is the {\it Hochschild homology} of $C$ (denoted by
$\mathrm{HH}_{\bullet} (C)$),
and the cyclic homology of the mixed complex $(\overline{\mathrm{CH}}_{\bullet} (C), b, B)$
is the {\it cyclic homology} of $C$.
\end{definition}

We leave to the interested readers to check that in the above definition
$b^2=B^2=bB+Bb=0$.


\begin{proposition}\label{Alg_Coalg}
Suppose that $C$ is a finite dimensional coalgebra.
Let $A:=\mathrm{Hom}_{\k}(C,\k)$ be its dual algebra.
Then
$$\mathrm{HC}_{\bullet}^{-}(C)\simeq\mathrm{HC}^{-{\bullet}}(A).$$
\end{proposition}

\begin{proof}
This is because
$$\mathrm{Hom}(A^{\otimes n}, \k)\simeq
\mathrm{Hom}(\mathrm{Hom}(C,\k)^{\otimes n},\k)
\simeq C^{\otimes n},$$
and thus as complexes,
\begin{equation*}\mathrm{CC}^{-{\bullet}}(A)\simeq \mathrm{CC}_{\bullet}^{-}(C),\end{equation*}
where the negative sign in the superscript $\mathrm{CC}^{-{\bullet}}(A)$
appears since we have to change the homological degree to the cohomological one.
\end{proof}

\begin{remark}
The definition of the Hochschild and cyclic homologies
can be generalized to DG algebras and DG coalgebras,
where we add in the boundary map $b$
the differential of the corresponding algebra and/or coalgebra respectively,
and the sign follows the Koszul sign convention.
\end{remark}

\section{Basics of Koszul algebras}\label{Sect_Koszul}

In this section, we recall some basics of Kosuzl algebras.
The notion of Koszul algebras is first introduced by Priddy \cite{Priddy};
for a comprehensive discussion of them, the reader may refer to Loday-Vallette \cite{LV}.

\subsection{Koszul complexes and Koszul algebras}

Suppose $V$ is a  finite dimensional    (possibly graded) vector space,
and let $T(V)$ be the free tensor algebra generated by $V$.
Let $R\subset V\otimes V$ be a subspace (graded subspace if $V$¡¡ is¡¡ graded).
We call the pair $(V, R)$ a quadratic datum.
The quadratic  algebra associated to this datum is defined to be  $$A=A(V, R):=T(V)/(R),$$
where $(R)$ be
the two sided ideal generated by $R$ in $T(V)$.  That is,
\begin{equation}\label{decomposition for a Koszul algebra}
A=T(V)/(R)=\k\oplus V\oplus  \cdots\oplus
\frac{V^{\ot n}}{\sum_{i=1}^{n-2} V^{\ot i}\ot R\ot V^{\ot n-2-i}}\oplus \cdots.
\end{equation}

In the \textit{classical} grading for Koszul algebras (see \cite{BGS}),
one usually puts elements of
$$A_n=\frac{V^{\ot n}}{\sum_{i=1}^{n-2} V^{\ot i}\ot R\ot V^{\ot n-2-i}}$$
in degree $n$. In this paper, we shall use a bigrading on $A$
such that  an  element
$$v_1\ot \cdots\ot v_n\in \frac{V^{\ot n}}{\sum_{i=1}^{n-2} V^{\ot i}\ot R\ot V^{\ot n-2-i}}$$
has degree $|v_1|+\cdots +|v_n|$ and weight $n$.
For example, if $V$ is ungraded,  elements of $A$ only live in the
column $\mathrm{degree}=0$, so in this case, the relation between the
classical grading and the bigrading is
\begin{equation}\label{Comparison two gradings for A} A_n=A_{0n}, n\geq 0.\end{equation}
We shall use these two notations in the sequel.
\textit{For the simplicity, we shall assume from now on that $V$ is ungraded.}

The quadratic  coalgebra  $C(V, R)$ associated to the quadratic datum $(V, R)$
is defined to be the subcoalgebra of $T^c(V)$ such that for any subcoalgebra $C$
of $T^c(V)$ such that the composition $C\rightarrowtail T^c(V)\twoheadrightarrow V^{\ot 2}/R$
is zero, there is a unique coalgebra morphism $C\to C(V, R)$ making the following diagram
commutative $$\xymatrix{& C(V, R)\ar[dr] & \\ C \ar[ur] \ar[rr] & & T^c(V)}.$$
Therefore
$$C=C(V, R)=\k\oplus V\oplus R\oplus \cdots \oplus
\big(\bigcap_{i+2+j=n} V^{\ot i}\ot R\ot V^{\ot j}\big)\oplus \cdots.$$
We define the Koszul dual coalgebra $A^\ac$ to be $C(sV, S^2R)$,
so elements of $$\bigcap_{i+2+j=n} (sV)^{\ot i}\ot (s^2R)\ot (sV)^{\ot j}$$ has
weight $n$ and degree equal to $n$. Since $V$ is ungraded (and thus is put in degree zero),
elements of $A^\ac$ live in the diagonal $degree=weight$ and in the first quadrant. In the classical
grading for this coalgebra, one puts $A^\ac_0=\k, A^\ac_1=V$ and for $n\geq 2$,
$$A^\ac_n=\bigcap_{i+2+j=n} (sV)^{\ot i}\ot (s^2R)\ot (sV)^{\ot j},$$
so this is just our weight grading and the relation between the classical grading and the bigrading
\begin{equation}\label{Comparison two gradings for Koszul dual coalgebra}
A^\ac_n=A^\ac_{nn}, n\geq 0.
\end{equation}

Given a quadratic coalgebra $C=C(V, R)$,
we define its Koszul dual algebra $C^\ac$ to be
$$A(s^{-1}V, s^{-2}R)=T(s^{-1}V)/(s^{-2}R).$$
We verify easily that $(A^\ac)^\ac=A$ and $(C^\ac)^\ac=C$.

The Koszul dual algebra of the quadratic algebra $A=A(V, R)$ is the graded dual
of $A^\ac=C(sV, s^2R)$, so denote by $R^\perp$ the subspace of $V^*\ot V^*$
consisting of functions vanishing on $R$ via the natural map $V^*\ot V^*\to (V\ot V)^*$,
then $$A^!=A(s^{-1}V^*, s^{-2}R^\perp).$$
Elements of
$(s^{-1}V^*)^{\ot n}/\sum_{i+2+j=n} (s^{-1}V^*)^{\ot i}\ot (s^{-2}R^\perp)\ot (s^{-1}V^*)^{\ot j}$
has weight $-n$ and degree the original degree minus  $n$.
Since  $V$ is ungraded (and thus is put in degree zero),
elements of $A^!$ live in the diagonal $degree=weight$ and in the third quadrant.
In the usual grading for this algebra, one puts $A^!_0=\k, A^!_1=V$ and
for $n\geq 2$, $A^!_n=(V^*)^{\ot n}/\sum_{i+2+j=n} (V^*)^{\ot i}\ot (R^\perp)\ot (V^*)^{\ot j}$,
so the relation between the classical grading and the bigrading is
\begin{equation}\label{Comparison two gradings for Koszul dual algebra}
A^!_n=A^!_{-n, -n}, n\geq 0\end{equation}

As is easily seen,  the canonical map
$\kappa: A^\ac\twoheadrightarrow sV\simeq V\rightarrowtail A$
is a twisting morphism, so one can form the twisted tensor product
$A^\ac\ot_\kappa A$ or $A\ot_\kappa  A^\ac$.

 \begin{definition}\label{DefKoszul}$A$ is said to be a {\it Koszul algebra}
 if $A\ot_\kappa  A^\ac$ is a projective  resolution of $\k$ or equivalently $\kappa$
 is a Koszul morphism.\end{definition}

We provide a concrete description of this complex.
Choose a basis $\{e_i\}$ of $V$, and let $\{e_i^*\}$ be the dual basis of $V^*$.
Let $\displaystyle\sum_i e_i\otimes s^{-1} e_i^*$ acts on $A\otimes A^{\ac}$
by
$$
d(a\otimes f)=(-1)^{|f|}\sum_i a e_i\otimes f\cdot  s^{-1} e_i^*,
$$
then $d^2$ is automatically $0$.
The complex $(A\otimes A^{\ac},d)$ is called the {\it Koszul complex}
of the quadratic algebra $A$ (it is in fact our $A\ot_\kappa A^\ac$).

One immediately gets that if $A$ is Koszul, then $A^!$ is also Koszul. In fact the Koszul complex of $A^!$ is
just $A^!\ot_{\kappa^\vee}A^\vee$ which is the dual of the Koszul complex of $A$.

Let us consider the bar construction of $A=A(V, R)$.  As we suppose that $V$ is ungraded,
it is put in degree zero. Graphically, the bar construction
$\mathrm{B}(A)=\k\oplus s\ol{A}\oplus (s\ol{A})^{\ot 2}\oplus (s\ol{A})^{\ot 3}\oplus \cdots $
 has the following structure
 $$
 \xymatrixrowsep{0pt}
 \xymatrix{\vdots & &\vdots &\vdots &\vdots & \\
3 & &\frac{V^{\ot 3}}{V\ot R+R\ot V} & V\ot  \frac{V^{\ot 2}}{R}+\frac{V^{\ot 2}}{R} \ot V \ar[l] & V\ot V\ot V \ar[l]&  \\
2 & &\frac{V^{\ot 2}}{R} & V\ot V \ar[l]& &  \\
1 & & V& & & \\
0 & \k & & & & \\
\mbox{weight/degree}  & 0 & 1& 2& 3& \cdots  }$$
In this graph,  we did not draw nonzero terms.
From this diagram, we know that there is a natural injection
$i: A^\ac\rightarrowtail\mathrm{B}(A)$ such that for any $i\geq 0$,
$H_{ii}(\mathrm{B}(A))$ can be identified with $A^\ac_i=A^\ac_{ii}$.

 Now consider the cobar construction of $A^\ac$.   Graphically, the bar construction
 $$\Omega(A^\ac)=
 \k\oplus s^{-1}\ol{A^\ac}\oplus (s^{-1} \ol{A^\ac})^{\ot 2}\oplus (s^{-1} \ol{A^\ac})^{\ot 3}\oplus\cdots
 $$
 has the following structure
$$
 \xymatrixrowsep{0pt}
\xymatrix{\vdots &\vdots &\vdots & \vdots& \vdots& \\
3 & VVV &VR+RV\ar[l]& RV\cap VR\ar[l] & &  \\
2 &VV &R\ar[l] & & &  \\
1 & V& & & & \\
0 & \k & & & & \\
\mbox{weight/degree}  & 0 & 1& 2& 3& \cdots  }$$
In this graph, we abreviate the tensor product $V\ot W$ to $VW$, and did not draw nonzero terms.
From this diagram, we know that there is a natural surjection $q: \Omega(A^\ac)\twoheadrightarrow A$
such that for any $i\geq 0$,
$H_{0i}(\Omega(A^\ac))$ can be identified with $A_i=A_{0i}$.

It follows from the fundamental result about twisting morphisms that $A$ is a Koszul algebra if and only if
$q: \Omega(A^\ac)\twoheadrightarrow A$ is an quasi-isomorphism of
differential bigraded algebras
if and only if $i: A^\ac\rightarrowtail \mathrm{B}(A)$ is an quasi-isomorphism
of differential bigraded coalgebras.

This kind of reciprocity is also reflected in the following proposition and it follows from the adjunction
between bar and cobar construction.


\begin{proposition}\label{CAN}
Suppose that $A$ is a Koszul algebra. Denote by $A^{\ac}$ the Koszul dual coalgebra of $A$.
Then we have a commutative diagram
$$\xymatrixcolsep{4pc}
\xymatrix{
A^{\ac}\ar[r]^{\eta}\ar[rd]_i& \mathrm{B}\Omega (A^{\ac})\ar[d]^{\mathrm{B}(p)}\\
&\mathrm{B}(A)
}
$$
of quasi-isomorphisms of differential graded coalgebras.
\end{proposition}

\begin{remark}Let $A$ be a Koszul algebra.
Since $A$ is locally finite, its graded dual $A^\vee$ is a coalgebra.
We have four universal twisting morphisms
$$\pi: \mathrm{B}(A)\to A,\quad \pi'=\iota^\vee: \mathrm{B}(A^!)\to A^!,
\quad \iota: A^\ac\to \Omega(A^\ac),\quad \iota'=\pi^\vee: A^\vee\to \Omega(A^\vee),$$
two  Koszul morphisms
$$\kappa: A^\ac\to A \quad\mathrm{and}\quad \kappa^\vee: A^\vee\to  A^!, $$
two quasi-isomorphisms of differential bigraded coalgebras
$$i: A^\ac\to \mathrm{B}(A)\quad\mathrm{and}\quad q^\vee: A^\vee\to \mathrm{B}(A^!),$$
and two quasi-isomorphisms of differential bigraded algebras
$$q: \Omega(A^\ac)\to A\quad \mathrm{and}\quad i^\vee:\Omega(A^\vee)\to A^!,$$
which satisfy some compatibility conditions.
Furthermore, there are four algebra-coalgebra dualities
\begin{equation}\label{DualityBetweenBarCobarConstructionForKoszulAlgebra}
\mathrm{B}(A)^\vee\simeq \Omega(A^\vee), \quad \Omega(A^\vee)^\vee\simeq\mathrm{B}(A),
\quad \mathrm{B}(A^!)^\vee\simeq \Omega(A^\ac),\quad \Omega(A^\ac)^\vee\simeq \mathrm{B}(A^!).
\end{equation}
\end{remark}

\begin{corollary}\label{Res}
Suppose that $A$ is a Koszul algebra. Denote by $A^{\ac}$ its Koszul dual coalgebra. Then
\begin{enumerate}
\item[$(1)$]
there is an isomorphism of graded coalgebras
$$A^\ac\simeq \mathrm{Tor}^A_\bullet(\k, \k)$$ and an isomorphism of graded algebras
$$
A^{!}\simeq \mathrm{Ext}_A^{-{\bullet}}(\k,\k).
$$
\item[$(2)$]
the complex
$$
\cdots\longrightarrow A\otimes A^{\ac}_m\otimes A\stackrel{b}{\longrightarrow} A\otimes A^{\ac}_{m-1}\otimes A
\stackrel{b}\longrightarrow\cdots\stackrel{b}\longrightarrow
A\otimes A^{\ac}_0\otimes A\simeq A\otimes A\stackrel{\mathrm{mult}}{\longrightarrow} A,$$
where
\begin{equation}\label{DifferentialInKoszulBimoduleResolutionOfA}
b(a\otimes c\otimes a')=
\sum_i \Big(e_ia\otimes  e_i^* c\otimes a'+(-1)^{m}a\otimes c e_i^* \otimes  a' e_i\Big)
\end{equation}
for $a\otimes c\otimes a' \in A\otimes A^{\ac}_m\otimes A$,
gives a resolution of $A$ as an $A$-bimodule.
\end{enumerate}
\end{corollary}

\begin{proof} Since $\pi: \mathrm{B}(A)\twoheadrightarrow s\ol{A}\simeq \ol{A}\rightarrowtail A$ is a
Koszul morphism,  $A\ot_\pi \mathrm{B}(A)\ot_\pi A$ is a projetive bimodule resolution of $A$ as bimodules.
For the Koszul algebra, the map $i: A^\ac\rightarrowtail\mathrm{B}(A)$ is a quasi-isomorphism of
differential bigraded coalgebras  and $\kappa: A^\ac\twoheadrightarrow sV\simeq V\rightarrowtail A$
is a Koszul morphism. It is easy to see that  $\kappa=\pi\circ i$, that is $\kappa$ and $\pi$ are compatible.
The comparison lemma shows that $A\ot_\kappa A^\ac\ot_\kappa A\to A\ot_\pi \mathrm{B}(A)\ot_\pi A$ is a
quasi-isomorphism. As a consequence, the complex
$\k\ot_A (A\ot_\kappa A^\ac_\bullet\ot_\kappa A)\ot_A \k$ computes $\mathrm{Tor}_\bullet^A(\k, \k)$,
and there is an isomorphism of complexes
$A^\ac_\bullet \simeq \k\ot_A (A\ot_\kappa A^\ac_\bullet\ot_\kappa A)\ot_A \k$.
Via the isomorphism, we endow $\mathrm{Tor}_\bullet^A(\k, \k)$ with a structure of graded coalgebra.
Taking homology groups, we see an isomorphism of graded coalgebras
$A^\ac_\bullet\simeq \mathrm{Tor}_\bullet^A(\k, \k)$.
Recall that we have used the classical grading for the complexes.

 For the second statement of (1), one can take  the graded duals of both sides
 of the first isomorphisms and notice that the graded dual of
 $\mathrm{Tor}_\bullet^A(\k, \k)$ is just $\mathrm{Ext}^{-\bullet}_A(\k, \k)$.
 Alternatively, one checks that the complex
 $\mathrm{Hom}_A(\k\ot_A (A\ot_\pi \mathrm{B}(A)\ot_\pi A), \k)$ is
 isomorphic to the graded dual of $\mathrm{B}(A)$.
Note that the minus sign in the second isomorphism comes from our grading
 convention that $A^!$ is negatively bigrading.

For (2),  this complex is just $A\ot_\kappa A^\ac\ot_\kappa A$.
For another nice proof, see Kr\"ahmer \cite[Proposition 19]{Krahmer}.
\end{proof}

\section{Hochschild (co)homology of Koszul algebras}\label{Sect_HK}

As proved by Priddy \cite{Priddy}, Koszul duality greatly simplifies
the computations of Hochschild homology and cohomology groups.

\subsection{Hochschild homology via the Koszul complex}

Denote
$$
K'(A)=\{\cdots\longrightarrow A\otimes A^{\ac}_m\otimes
A\stackrel{b}{\longrightarrow} A\otimes A^{\ac}_{m-1}\otimes A
\stackrel{b}\longrightarrow\cdots\stackrel{b}\longrightarrow
A\otimes A^{\ac}_0\otimes A\}$$
be the projective bimodue  resolution of $A$ as in Corollary \ref{Res},
that is, $K'(A)=A\ot_\kappa A^\ac_\bullet\ot_\kappa A$. Let
$K(A):=A\otimes_{A\otimes A^{\mathrm{op}}} K'(A)$.
More concretely (recall that we assume that $V$ is ungraded),
choose a basis $\{e_i\}$ of $V$, and let $\{e_i^*\}$ be the dual basis of $V^*$. Then
$$
K(A)\simeq (A\otimes A^{\ac}_\bullet, b)
$$
with differential
\begin{equation}\label{DifferentialInKoszulComplex}
b(a\otimes c)=\sum_{i}
\Big(a e_i a \otimes c\cdot s^{-1}e_i^*  +(-1)^{m}  e_i a\otimes  s^{-1} e_i^* \cdot c\Big)
\end{equation}
for $a\otimes c\in A\otimes A^{\ac}_{m}$.
 Therefore, we have
\begin{equation}\label{HHo_Kos}
\mathrm{HH}_{\bullet} (A)\simeq\mathrm H_{\bullet}(K(A), b).
\end{equation}
Observe that the homological grading coincides with  the classical grading. Since $A$ is bigraded,
the normalized Hochschild complex of $A$  and  the Hochschild homology are also bigraded. Define
$$\mathrm{HH}_{ij}(A)=\mathrm{H}_{ij}(\ol{\mathrm{CH}}_{\bullet\bullet}(A)) \simeq \mathrm{H}_{ij}(K(A), b),$$
then the  relation between the homological grading and the bigrading is
\begin{equation}\label{ComparisonTwoGradingForHochschildHomologyOfA}
\mathrm{HH}_n(A)=\bigoplus_{i} \mathrm{HH}_{n, n+i}(A), n\geq 0.\end{equation}

Similarly, define
$$\mathrm{HH}_{ij}(A^\ac)=\mathrm{H}_{ij}(\ol{\mathrm{CH}}_{\bullet\bullet}(A^\ac))$$
and we obtain
\begin{equation}\label{ComparisonTwoGradingForHochschildHomologyOfA^ac}
\mathrm{HH}_n(A^\ac)=\bigoplus_{i} \mathrm{HH}_{i, i+n}(A), n\geq 0.\end{equation}

\begin{theorem}\label{can_iso}
Let $A$ be a Koszul algebra, and denote by $A^{\ac}$ its Koszul dual coalgebra.
Then we have bigraded isomorphisms
$$\mathrm{HH}_{{\bullet}{\bullet}}(A)\simeq\mathrm{H}_{{\bullet}{\bullet}}(K(A),b)
\simeq\mathrm{HH}_{{\bullet}{\bullet}}(A^{\ac}),$$
which respects the Connes cyclic operator on both sides.
\end{theorem}

\begin{lemma}\label{Main_Lemma2}
Let $A$ be a Koszul algebra, and $A^{\ac}$ be its Koszul dual coalgebra. Then
we have commutative diagram of quasi-isomorphisms of $b$-complexes
$$
\xymatrix{
&\ol{\mathrm{CH}}_\bullet(\Omega(A^\ac))=\Omega (A^{\ac})\otimes\mathrm B\Omega (A^{\ac})\ar[ld]_{p_1}&\\
\ol{\mathrm{CH}}_\bullet(A)=A\otimes \mathrm B(A)&&\ol{\mathrm{CH}}_\bullet(A^\ac)
=\Omega (A^{\ac})\otimes A^{\ac}.\ar[lu]_{q_2}\ar[ld]^{\phi_2}\\
&K(A)=A\otimes A^{\ac}\ar[lu]^{\phi_1}&
}
$$
\end{lemma}

\begin{proof}
We have explicit formulas for $p_1,q_2,\phi_1$ and $\phi_2$:
\begin{eqnarray*}
p_1=\overline{\mathrm{CH}}_\bullet(p)=p\ot \mathrm{B}(p); && q_2=id\otimes \eta;\\
\phi_1=id\otimes i;&&\phi_2= p\otimes id,
\end{eqnarray*}
where  $\eta, i$ are given in Proposition \ref{CAN}.
One verifies easily that these are chain maps and
the  commutativity of the diagram follows from  Proposition~\ref{CAN}.

We already know that  $\phi_1$ is a quasi-isomorphism.
Since the functor $\overline{\mathrm{CH}}_{\bullet}(-)$ preserves quasi-isomorphisms, $p_1$ is a quasi-isomorphism (\cite[5.3.5 Theorem]{Loday}).

The proof of other two    quasi-isomorphisms follows from a standard spectral sequence argument.
To show
the quasi-isomorphism $$\phi_2: \Omega(A^\ac)\otimes A^{\ac}\to A\otimes  A^{\ac},$$
filter both complexes by
$$ F_r(\Omega(A^\ac)\otimes A^{\ac})
=\bigoplus_{s\leq r} \Omega(A^{\ac})\otimes A^{\ac}_{s},
\quad
 F_r(A\otimes A^{\ac})
=\bigoplus_{s\leq r} A\otimes  A^{\ac}_{s},
$$
then both of the boundary maps respect the filtration,
and the comparison theorem for spectral sequences (\cite[Lemma 2.5.1]{LV})
guarantees the quasi-isomorphism, induced by the quasi-isomorphism
$p: \Omega(A^{\ac})\to A^{\ac}$.
For $q_2: \Omega(A^{\ac})\otimes A^{\ac}\to \Omega(A^\ac)\otimes A^{\ac} $,
just choose the filtrations
\begin{equation*}
F_r(\Omega(A^\ac)\otimes A^{\ac})
=\bigoplus_{s\leq r} \Omega(A^{\ac})\otimes A^{\ac}_{s},
\quad F_r(\Omega (A^{\ac})\otimes\mathrm B\Omega (A^{\ac}))=\bigoplus_{s\leq r} \Omega (A^{\ac})\otimes\mathrm B\Omega (A^{\ac})_{{\bullet}, s}.
\qedhere
\end{equation*}

%
%
%
%
\end{proof}

Now let us recall a result of Loday-Quillen
(see Loday-Quillen \cite[\S5]{LQ} and Loday \cite[\S3.1]{Loday})
that for a free algebra $T(V)$ generated by $V$,
its Hochschild and cyclic homology can be computed by the {\it small complex}
$$
T(V)\otimes (\k\oplus s V).
$$
The differential of this small complex is given as follows:
If we view $T(V)$ as the cobar construction of the coalgebra $\k\oplus s  V$
with trivial reduced coproduct, then the differential is exactly the Hochschild boundary
map for coalgebras (see (\ref{Hochbcopx_coalg})). This result is generalized to the case of
the cobar construction $\Omega(C)$
of any coaugmentted DG coalgebra $C$ by Vigu\'e-Poirrier \cite{V} and Jones-McCleary \cite{JM},
that is, the Hochschild and cyclic homology of $\Omega(C)$ can also be
computed via this small complex, {\it i.e.}
$$(\overline{\mathrm{CH}}_{\bullet}(\Omega(C)),b,B)\simeq(\ol{\mathrm{CH}}_\bullet (C), b, B).$$
It is summarized into the following lemma.

\begin{lemma}\label{Main_Lemma1}
Let $A$ be a Koszul algebra, and denote by $A^{\ac}$ its Koszul dual coalgebra.
We have quasi-isomorphisms
of mixed complexes
$$
\xymatrix{
&\overline{\mathrm{CH}}_{\bullet}(\Omega(A^\ac))=\Omega (A^{\ac})\otimes
\mathrm B\Omega (A^{\ac})\ar[ld]_{p_1}\ar[rd]^{p_2}&\\
\overline{\mathrm{CH}}_{\bullet}(A)=A\otimes \mathrm B(A)&&
\overline{\mathrm{CH}}_{\bullet}(A^\ac)=\Omega (A^{\ac})\otimes A^{\ac}.
}
$$
where $p_2$ is a homotopy inverse of $q_2$ in Lemma \ref{Main_Lemma2}.
\end{lemma}

\begin{proof}
(1) Since $p: \Omega(A^{\ac})\to A$ is a quasi-isomorphism of DG algebras,
by applying the Hochschild chain complex functor, we obtain
$$p_1=\overline{\mathrm{CH}}_\bullet(p):\overline{\mathrm{CH}}_{\bullet}(\Omega(A^{\ac}))\to
\overline{\mathrm{CH}}_{\bullet}(A)$$
is a quasi-isomorphism of mixed complexes.

(2) Existence of $p_2$:
The formula for $p_2$ is given as follows.
Write elements in the bar construction in the form $[u_1|u_2|\cdots|u_n]$
and elements in the cobar construction in the form $\langle a_1|a_2|\cdots |a_n\rangle $.
Define $p_2:\overline{\mathrm{CH}}_{\bullet} (\Omega(A^{\ac}))\to
\overline{\mathrm{CH}}_{\bullet} (A^{\ac})$ by
\begin{eqnarray*}
\Omega(A^{\ac})\otimes \mathrm B\Omega(A^{\ac})&\longrightarrow&\Omega(A^{\ac})\otimes A^{\ac}\\
\langle a_1|a_2|\cdots |a_n\rangle \otimes 1&\longmapsto&\langle a_1|a_2|\cdots |a_n\rangle \otimes 1\\
\langle a_1|a_2|\cdots |a_n\rangle\otimes[\langle u_1|u_2|\cdots |u_m\rangle]&\longmapsto&\sum_i (-1)^{\mu_i}
\langle u_{i+1}|\cdots |u_m|a_1|\cdots |a_n |u_1|\cdots |u_{i-1}\rangle \otimes u_i\\
\langle a_1|a_2|\cdots |a_n|\rangle \otimes[v_1|\cdots|v_r]&\longmapsto&0,
\end{eqnarray*}
where $a_1, \cdots, a_n, u_1, \cdots, u_m\in \ol{A}$, $v_1, \cdots, v_r\in \Omega(A^{\ac})$ with $r>1$
and $\mu_i=(|u_{i+1}|+\cdots+|u_m|)(|a_1|+\cdots +|a_n|+|u_1|+\cdots+|u_{i}|+1)$.
The reader may find in Vigu\'e-Poirrier \cite{V} and Jones and McCleary \cite[\S6]{JM}
(where $p_2$ is denoted by $\overline\jmath$)
that
$$p_2:(\overline{\mathrm{CH}}_{\bullet}(\Omega(A^{\ac})), b, B)
\to(\overline{\mathrm{CH}}_{\bullet}(A^{\ac}), b, B)$$
is a morphism of mixed complexes.

It is direct to see that
$$p_2\circ q_2=id: \Omega(A^{\ac})\otimes A^{\ac}\to\Omega(A^{\ac})\otimes A^{\ac}.$$
Therefore $p_2$ is an quasi-isomorphism of $b$-complexes.

From homological algebra (see, for example, Loday \cite[Proposition 2.5.15]{Loday})
we know that for a morphism of mixed complexes, if it is a quasi-isomorphism
for the $b$-differential, then the morphism is in fact a quasi-isomorphism of mixed complexes.
The lemma is now proved.
\end{proof}

\begin{proof}[Proof of Theorem \ref{can_iso}]
The statement follows directly from Lemmas \ref{Main_Lemma2} and \ref{Main_Lemma1}.
\end{proof}

\begin{remark} Recall that we have
$$\mathrm{HH}_n(A)=\bigoplus_{i} \mathrm{HH}_{n, i+n}(A)$$ and
$$\mathrm{HH}_n(A^\ac)=\bigoplus_{i}\mathrm{HH}_{i, i+n}(A^\ac).$$
The isomorphism in Theorem~\ref{can_iso} preserves the bigrading and it does
NOT follow that $\mathrm{HH}_n(A)\simeq \mathrm{HH}_n(A^\ac)$!
\end{remark}

\subsection{Hochschild cohomology via the Koszul complex}

Now, let us go to Hochschild cohomologies.
Recall that $\mathrm{HH}^{\bullet}(A)\simeq\mathrm{Ext}^{\bullet}_{A^e}(A,A)$.
For Koszul algebras, we have
\begin{eqnarray}
\mathrm{HH}^i(A)
&=&\mathrm H_{-i}(\mathrm{Hom}_{A^e}(A\otimes_\kappa A^{\ac}_\bullet\otimes_\kappa A, A))\nonumber\\
&=&\mathrm H_{-i}(\mathrm{Hom}^\kappa(A^{\ac}_\bullet, A))\nonumber\\
&\stackrel{(\ref{TransferingHomToTensorProduct})}{=}&\mathrm H_{i}(A\otimes^\kappa A^!_{\bullet}),\label{HH_Kos}
\end{eqnarray}
where the differential of the  complex $A\otimes^\kappa A^!_\bullet$ is given
\begin{equation}\label{DifferentialInComplexComputingHochschildCohomologyKoszulAlgebra}
\delta(a\ot x)=\sum_i \Big(
  e_i a \ot e_i^* x  +(-1)^{|x|}  a  e_i  \ot xe_i^*\Big).\end{equation}

Notice that the cohomological grading coincides with the classical grading.
Since we work in the bigraded setup, we define
$$\mathrm{HH}^{ij}(A)\simeq\mathrm{H}_{-i, -j} (A\otimes^\kappa A^!)$$ and
$$\mathrm{HH}^{ij}(A^!)\simeq\mathrm{H}_{-i, -j} (A^!\otimes^{\kappa^\vee} A).$$ We obtain that
\begin{equation}\label{ComparisonTwoGradingForHochschildCoHomologyOfA}
\mathrm{HH}^n(A)=\bigoplus_{i} \mathrm{HH}^{n, n-i}(A), n\geq 0.\end{equation}
and \begin{equation}\label{ComparisonTwoGradingForHochschildCoHomologyOfA!}
\mathrm{HH}^n(A^!)=\bigoplus_{i} \mathrm{HH}^{i, -n+i}(A^!), n\geq 0.\end{equation}

Since for Koszul algebras, $(A^!)^!=A$,   we have a natural isomorphism:

\begin{theorem}[Buchweitz \cite{Buchweitz}; Beilinson-Ginzburg-Soergel \cite{BGS}; Keller \cite{Keller}]
\label{G-alg_pers}
Let $A$ be a Koszul algebra and $A^!$ be its Koszul dual algebra.
Then there are natural isomorphisms of bigraded algebras
$$\mathrm{HH}^{\bullet\bullet}(A) \stackrel{\simeq}\longrightarrow
\mathrm{HH}^{\bullet\bullet}(A^!), $$
 where the product on both
sides are the  cup product.
\end{theorem}

\begin{proof}
We have the following diagram
$$
\xymatrix{\ol{\mathrm{CH}}^{\bullet }(A;A)\ar[r]^-{\simeq}
&\mathrm{Hom}^\pi(\mathrm B(A), A)\ar[r]_-{\simeq}^-{(\ref{TransferingHomToTensorProduct})}
&A\ot^\pi \Omega(A^\vee)
\ar[r]^-{(\ref{InducedMorphismBetweenNewTwistedTensorProduct})}_-{\simeq}& A\ot^\kappa A^!\ar[d]^{(\ref{DualTwistingMorphismsInducingIsomorphism})}
\\
\ol{\mathrm{CH}}^{\bullet }(A^!;A^!)\ar[r]^-{\simeq}
&\mathrm{Hom}^{\iota^\vee}(\mathrm B(A^!), A^!)\ar[r]_-{\simeq}^-{(\ref{TransferingHomToTensorProduct})}
&A^!\ot^{{\iota^\vee}}\Omega(A^\ac)
\ar[r]^-{(\ref{InducedMorphismBetweenNewTwistedTensorProduct})}_-{\simeq}&A^!\ot^{\kappa^\vee} A.
}$$
Each map in the diagram is an isomorphism of DG algebras.
This completes the proof. Notice that by the following commutative diagram
$$\xymatrix{
  \Omega(A^\ac)\ot^\iota A^\ac \ar[d]_{(\ref{DualTwistingMorphismsInducingIsomorphism})}^-\simeq
\ar[r]^-{(\ref{InducedMorphismBetweenNewTwistedTensorProduct})}_-{\simeq}&
A\ot^\kappa A^!\ar[d]_{(\ref{DualTwistingMorphismsInducingIsomorphism})}^-\simeq
\\
 A^!\ot^{{\iota^\vee}}\Omega(A^\ac)
 \ar[r]^-{(\ref{InducedMorphismBetweenNewTwistedTensorProduct})}_-{\simeq}&A^!\ot^{\kappa^\vee} A.
}$$
we have a sequence of isomorphisms of differential bigraded algebras
\begin{equation}\label{IsomorphismsHochschildCohomologyKoszul}
\ol{\mathrm{CH}}^{\bullet }(A;A) \simeq A\ot^\kappa A^!\simeq
\Omega(A^\ac)\ot^\iota A^\ac\simeq \ol{\mathrm{CH}}^{\bullet }(A^!;A^!).
\qedhere
\end{equation}
\end{proof}

\begin{remark} As in the homology case,
the above isomorphism preserves the bigrading, but not the usual grading.
\end{remark}

\begin{remark}
We learned from Keller \cite{Keller} that Buchweitz \cite{Buchweitz}
first proves the above isomorphism as graded algebras, while in
the same paper Keller proves that this isomorphism is in fact
an isomorphism of Gerstenhaber algebras, where the Gerstenhaber
Lie bracket can be interpreted as the Lie algebra of the derived Picard group of the algebra.
A little bit earlier than Keller,
Beilinson-Ginzburg-Soergel \cite{BGS} prove that
the isomorphism is an isomorphism of graded associative algebras.
The proof of Keller and Beilinson-Ginzburg-Soergel uses
derived categories, where they first show that  certain derived categories
of a Koszul algebra and its Koszul dual are derived equivalent,
and then show that the Hochschild cohomology is an invariant of the derived category.
\end{remark}

\subsection{The linear-quadratic Koszul algebras}
In the original work of Priddy \cite{Priddy}, being Koszul
means linear-quadratic Koszul, where being Koszul in the sense of Definition \ref{DefKoszul}
is a special case.

As before, $V$ is a finite dimensional vector space.
Let $R\subset V\oplus V^{\otimes 2}$ and consider
$$A:=T(V)/(R),$$
which is called a {\it linear-quadratic algebra}. Without loss of generality
we may assume $R\cap V=\{0\}$ (by removing those elements in $R\cap V$,
this can be satisfied without affecting $A$).
Under this assumption, if we denote
$qR$ to be the image of the projection of $R$ to $V^{\otimes 2}$,
we obtain a map
$$\phi: qR\to V$$
such that $R=\{X-\phi(X)|X\in qR\}$.
Denote by $(qA)^{\ac}$ the quadratic dual coalgebra of $T(V)/(qR)$, then this $\phi$ gives a map
$$d_\phi: (qA)^{\ac}\twoheadrightarrow qR\to V,$$
which extends to a coderivation $d_\phi: (qA)^{\ac}\to T(V)$.

Now if
\begin{equation*}
\{R\otimes V+V\otimes R\}\cap R^{\otimes 2}\subset qR,
\end{equation*}
then the images of $d_\phi$ lie in $(qA)^{\ac}$. We in fact get
a co-derivation
$$d_\phi: (qA)^{\ac}\to (qA)^{\ac}.$$
And if furthermore,
\begin{equation}\label{linear-quadratic}
\{R\otimes V+V\otimes R\}\cap R^{\otimes 2}\subset R\otimes V^{\otimes 2},
\end{equation}
then $(d_\phi)^2=0$,
and we obtain a DG coalgebra $((qA)^{\ac}, d_\phi)$.
For more details, see \cite[\S3.6]{LV}.

\begin{definition}[Linear-quadratic Koszul algebras]\label{strict_Koszul}
Let $V, R$ and $A$ be as above. $A$ is called a {\it linear-quadratic Koszul algebra}
if $R$ satisfies (\ref{linear-quadratic}) and the associated $T(V)/(qR)$ is Koszul.
\end{definition}

\begin{example}[Universal enveloping algebras]
Suppose $\mathfrak g$ is a Lie algebra over $\k$,
then the universal enveloping algebra
$U(\mathfrak g)$ is a linear-quadratic Koszul algebra, whose Koszul dual DG coalgebra
is the Chevalley-Eilenberg chain complex $\mathrm C_{\bullet}(\mathfrak g)$ of $\mathfrak g$.
\end{example}

The reader may check that almost all statements about Koszul algebras in
the above and below sections also
hold for linear-quadratic Koszul algebras (except that
all differentials involved now have an extra term coming
from the differential of the Koszul dual coalgebra).

\begin{remark}\label{general_Koszul}
In most literature, an algebra being Koszul is in the sense of Definition \ref{DefKoszul}
or \ref{strict_Koszul}.
In the introduction, we mentioned Koszul duality in the general sense,
where the Koszul dual of an algebra is defined to be
$\mathrm{Ext}^{\bullet}_A(\k,\k)$ (equipped with the associated $A_\infty$
structure), or even more generally,
the dual space of its bar construction. Koszul duality in this general sense has been used in, for
example, \cite{LPWZ} and \cite{VdB3}.
\end{remark}

\section{Koszul Calabi-Yau algebras}\label{Sect_KCY}

In the following, for an associative algebra $\Lambda$, by $\mathcal{D}(\Lambda)$ we
mean the derived category of left $\Lambda$-modules and $\Sigma$
denoted the translation functor in the derived category.

\subsection{Definition of Calabi-Yau algebras}
\begin{definition}[Ginzburg \cite{Ginzburg}]\label{Def_CY}
A (bi-)graded algebra $A$ is said to be \textit{Calabi-Yau of dimension $n$} (or $n$-Calabi-Yau for short)
if $A$ is homologically smooth and
there exists an isomorphism
\begin{equation}\label{Iso_CY}
\eta: \mathrm{RHom}_{A^e}(A, A\ot A)\longrightarrow\Sigma^{n} A(\ell)
\end{equation}
in the derived category of (bi-)graded left bimodules $\mathcal{D}(A^e)$,
where being homologically smooth means $A$ is a perfect $A^e$-module,
{\it i.e.}, $A$ admits a bounded resolution of finitely generated projective (bi-)graded
$A^e$-modules, and $n, \ell$ are integers.
\end{definition}

Notice that in the above definition, $A\ot A$ is considered as a left $A^e$-module,
using the outer structure, that is, $(a\ot b^{op})\cdot (x\ot y)=(-1)^{|b|(|x|+|y|)}ax\ot yb$
for $x\ot y\in A\ot A$ and $a\ot b^{op}\in A^e$; it is  thus an object in $\mathcal{D}(A^e)$
and the expression $\mathrm{RHom}_{A^e}(A, A\ot A)$ makes sense. It has another
$A^e$-module structure using the internal structure, that is,
$(a\ot b^{op})\cdot (x\ot y)=(-1)^{|a||b|+|a||x|+|x||b|}xb\ot ay$,
which induces a left $A^e$-module structure over
$\mathrm{RHom}_{A^e}(A, A\ot A)$, which says that the isomorphism (\ref{Iso_CY}) has a meaning.

In the original definition of Ginzburg the isomorphism $\eta$ is required to be self dual, which is proved
by Van den Bergh (\cite[Appendix C]{VdB3}) to be automatically satisfied. However, on the other
hand, the isomorphism
(\ref{Iso_CY}) may not be unique.


Suppose $A$ is Koszul. Recall that \begin{equation}\label{Res_A}
K'(A)=A\otimes_\kappa A^{\ac}_\bullet\otimes_\kappa A=\{
\cdots \longrightarrow A\otimes A^{\ac}_1\otimes A\longrightarrow A\otimes A^{\ac}_0\otimes A\}.
\end{equation}
is free resolution of $A$ as an $A^e$-module, and therefore
\begin{eqnarray}
\mathrm{RHom}_{A^e}(A, A\ot A)&\simeq&
\mathrm{Hom}_{A^e}(A\otimes_\kappa A^{\ac}_\bullet\otimes_\kappa A, A\ot A)\nonumber\\
&\simeq&\mathrm{Hom}_\k(A^{\ac}_\bullet, A\ot A)\nonumber\\
&\simeq&(A\otimes A^!_{-\bullet}\otimes A, \partial)\label{Ko_Res}
\end{eqnarray}
in $\mathcal{D}(A^e)$. Notice that in the third isomorphism,
we changed the position of two $A$'s, so the original internal structure becomes the outer
structure and the differential in the complex $(A\otimes A^!_{-\bullet}\otimes A, \partial)$ is given by
\begin{equation}\label{DifferentialForComplexComputingKoszulCalabi-Yau}
\partial(1\ot a\ot 1)=\sum_j(e_j\ot e_j^*a\ot 1+(-1)^{|a|} 1\ot ae_j^*\ot e_j),
\end{equation}
where  $\{e_i\}$ is a chosen basis of $V$, and   $\{e_i^*\}$ is the dual basis of $V^*$.

Now if furthermore $A$ is Calabi-Yau, then
\begin{equation}\label{Frob_id}
A\otimes A^!_{-\bullet}\otimes A\stackrel{(\ref{Ko_Res})}\simeq
\mathrm{RHom}_{A^e}(A, A\ot A)
\stackrel{(\ref{Iso_CY})}\simeq\Sigma^{n} A(\ell)
\stackrel{(\ref{Res_A})}\simeq \Sigma^{n}(A\otimes_\kappa A^{\ac}(\ell)_\bullet\otimes_\kappa A)
\end{equation}
in $\mathcal{D}(A^e)$, which implies
\begin{equation}
\label{KD}
\psi: A^!_{-\bullet}\simeq A^{\ac}_{\bullet-n}(\ell)
\end{equation}
in $\mathcal{D}(\k)$.

\begin{remark} The isomorphism (\ref{KD}) shows that for $i\geq 0$,
$A^!_i=A^!_{-i, -i}$ lying in complete degree $(-i, -i)$ is isomorphic to
$A^\ac_{n-i}=A^\ac_{n-i, n-i}$
placed in complete degree $(n-i-\ell, n-i-\ell)$.
Since all the isomorphisms in (\ref{Frob_id}) preserve the bigrading,
we obtain that $n=\ell$. This means that for (bi)graded Koszul Calabi-Yau algebra,
the Calabi-Yau dimension $n$ is equal to the shift $\ell$ in Definition~\ref{Def_CY}.
This fact, well-known to the expert, seems not  have been written out explicitly.

\end{remark}

The isomorphism (\ref{KD}) leads to the following result,
which is originally due to Van den Bergh \cite{VdB0,VdB3}:

\begin{proposition}[Van den Bergh]\label{PDVdB}
Suppose that $A$ is a Koszul algebra.
Then $A$ is $n$-Calabi-Yau
if and only if $A^{!}$ is a cyclic algebra of degree $n$.
\end{proposition}

Recall that a (bi-)graded associative algebra $A$ is {\it cyclic} of degree $n$
if it admits a   non-degenerate bilinear pairing
$\langle-,-\rangle: A\times A\to \k(n)$
such that $$\langle a\cdot b,c\rangle=(-1)^{(|a|+|b|)|c|}\langle c\cdot  a,b\rangle,\quad\mbox{for}\;\;
a,b,c\in A.$$
If $A$ is a DG algebra, then the pairing should furthermore satisfy
$$\langle d(a),b\rangle+(-1)^{|a|}\langle a,d(b)\rangle=0.$$

\begin{proof}[Proof of Proposition \ref{PDVdB}]
It is well-known that $$A^!_\bullet\simeq  A^{\ac}_{n-\bullet}(n)$$
as $A^!$-bimolules is equivalent
to that $A^!$ is cyclic (see, for example, Rickard \cite[Theorem 3.1]{Rickard}).
Therefore, we only need to show the isomorphism given by (\ref{KD})
is compatible with the $A^!$ actions (see, for example, Smith \cite[Proposition 5.10]{Smith}).
This is true because the resolution $A\otimes_\kappa A^{\ac}\otimes_\kappa A$ of
$A$ as $A^e$-module is minimal
(the differential has no linear terms), which is then unique up to isomorphism, and
the differential corresponds to the multiplication of generators of $A^!$
on $A^!$ and $A^{\ac}$ respectively.
More precisely,
for $a\in A^!_i$,
\begin{eqnarray*}b(\mathrm{Id}_A\ot \psi\ot \mathrm{Id}_A)(1\ot a\ot 1)
&=&b(1\ot \psi(a)\ot 1)\\
&=& \sum_j(e_j\ot e_j^*\psi(a)\ot 1+(-1)^i 1\ot \psi(a)e_j^*\ot e_j),\end{eqnarray*}
and
\begin{eqnarray*}
(1_A\ot \psi\ot 1_A)\partial(1\ot a\ot 1)
&=&(1_A\ot \psi\ot 1_A)\Big(\sum_j(e_j\ot e_j^*a\ot 1+(-1)^i 1\ot ae_j^*\ot e_j)\Big)\\
&=& \sum_j(e_j\ot \psi(e_j^*a)\ot 1+(-1)^i 1\ot \psi(ae_j^*)\ot e_j).\end{eqnarray*}
That is, we
have the following commutative diagram
$$
\xymatrixcolsep{5pc}
\xymatrix{ A\ot A^!_i\ot A \ar[r]^-{(\ref{Frob_id})}_{1_A\ot \psi\ot 1_A}
\ar[d]^\partial_{(\ref{DifferentialForComplexComputingKoszulCalabi-Yau})}
& A\ot_\kappa A^\ac_{n-i}\ot_\kappa A\ar[d]_{(\ref{DifferentialInKoszulBimoduleResolutionOfA})}^b\\
A\ot A^!_{i+1}\ot A\ar[r]^-{(\ref{Frob_id})}_{1_A\ot \psi\ot 1_A}& A\ot_\kappa A^\ac_{n-i-1}\ot_\kappa A.}
$$
This shows that $\psi$ is an isomorphism of $A^!$-bimodules, which completes the proof.
\end{proof}

\subsection{The noncommutative Poincar\'e duality}

The noncommutative Poincar\'e duality, due to Van den Bergh,
arises from his study of the dualising complexes
in noncommutative projective geometry (\cite{VdB,VdB0,VdB1}).
Lambre understands it from the viewpoint of differential calculus with duality (\cite{Lambre}).
Recently,
de Thanhoffer de V\"{o}lcsey and
Van den Bergh \cite[Proposition 5.5]{VdBdTdV12} give an
explicit formula for Van den Bergh's Poincar\'e duality
for Calabi-Yau algebras:

\begin{theorem}[Poincar\'e duality of Van den Bergh]
Let $A$ be an $n$-Calabi-Yau algebra.
Then there is an isomorphism
$$
\mathrm{PD}:\mathrm{HH}^i(A;A)\stackrel{\simeq}\longrightarrow\mathrm{HH}_{n-i}(A)
$$
for each $i$.
\end{theorem}

\begin{proof}[Proof of the Koszul case]

For the Koszul case, by comparing the differentials  (\ref{DifferentialInKoszulComplex})
and (\ref{DifferentialInComplexComputingHochschildCohomologyKoszulAlgebra}),
since  (\ref{KD}) is an isomorphism of $A^!$-bimodules, it
 induces an isomorphism of complexes
\begin{equation}\label{IsomorphismOfComplexesInducedByKD}
(A\ot^\kappa A^!,  \delta)\simeq \Sigma^{-n}(A\ot A^\ac, b),\end{equation}
and so  we have
\begin{equation}\label{DualityIsomorphisomInBigradingForHochschildCohomology}
\mathrm{HH}^{ij}(A;A)
\stackrel{(\ref{HH_Kos})}{\simeq}
\mathrm H_{-i, -j}(A\ot A^{!},\delta)
\stackrel{(\ref{IsomorphismOfComplexesInducedByKD})}{\simeq}  \mathrm H_{n-i, n-j}(A\ot A^{\ac},b)
=\mathrm{HH}_{n-i, n-j}(A).
\end{equation}
This is an isomorphism with respect to the bigrading. For the classical grading, we have
$$\mathrm{HH}^i(A)
\stackrel{(\ref{ComparisonTwoGradingForHochschildCoHomologyOfA})}
{\simeq}\bigoplus_{j} \mathrm{HH}^{i, i-j}(A) \stackrel{(\ref{DualityIsomorphisomInBigradingForHochschildCohomology})}
{\simeq} \bigoplus_j \mathrm{HH}_{n-i, n-i+j}(A) \stackrel{(\ref{ComparisonTwoGradingForHochschildHomologyOfA})}
{\simeq} \mathrm{HH}_{n-i}(A).  $$
This completes the proof.
\end{proof}

\subsection{Example of the universal enveloping algebras}

Recall that a finite dimensional Lie algebra $\mathfrak g$ is called {\it unimodular}
if the traces of the adjoint actions are zero, {\it i.e.} $\mathrm{Tr}(\mathrm{ad}_g(-))=0$,
for all $a\in\mathfrak g$.
Examples of unimodular Lie algebras are semi-simple Lie algebras, Heisenberg Lie algebras,
Lie algebras of compact Lie groups, etc. However, not
all Lie algebras are unimodular; for example, consider $\mathrm{Span}_{\k}\{x,y\}$ with $[x, y]=x$, it is not
unimodular.

It is nowadays well-known that the universal enveloping algebra $U(\mathfrak g)$ of a unimodular Lie algebra
$\mathfrak g$ is Calabi-Yau ({\it cf}. \cite[Proposition 3.3]{HVZ}).
In the following we give a simplified proof of this fact by using Koszul duality.

Assume $\mathfrak g$ is an $n$-dimensional Lie algebra over $\k$.
The Koszul dual algebra and coalgebra of $U(\mathfrak g)$ are the Chevalley-Eilenberg
cochain complex
$(\mathrm C^{\bullet}(\mathfrak g), \delta)$ and chain complex $(\mathrm C_{\bullet}(\mathfrak g),\partial)$,
respectively.
Choose a nonzero element $\Omega\in\mathrm C_{n}(\mathfrak g)$, then we have an
isomorphism
$$
\begin{array}{cccl}
\psi:&\mathrm C^i(\mathfrak g)&\longrightarrow& \mathrm C_{n-i}(\mathfrak g)\\
&f&\longmapsto& f\cap\Omega,
\end{array}
$$
of vector spaces, for $i=0,1,\cdots, n$.

\begin{lemma}\label{PD_unimodular}
Let $\mathfrak g$ be an $n$-dimensional
Lie algebra. Then
$$\psi: \mathrm C^{\bullet}(\mathfrak g)\longrightarrow\mathrm C_{n-{\bullet}}(\mathfrak g)$$
defined above is an isomorphism of chain complexes if and only if $\mathfrak g$ is unimodular.
\end{lemma}

\begin{proof}
Define the intersection product
$ \mathrm C_i(\mathfrak g)\times \mathrm C_{n-i}(\mathfrak g)\to \k$
by
$$
(u, v)\mapsto\langle u,v\rangle:= u\wedge v/\Omega,
$$
where the right-hand side means the scalar multiplicity of $u\wedge v$ with respect to $\Omega$.
That $\psi$ is an isomorphism of chain complexes is
equivalent to that the intersection product respects the differential.
In fact, without loss of generality,
we may assume $u= g_1\wedge\cdots\wedge g_i$, $v=  g_i\wedge\cdots\wedge g_n$.
Then
\begin{eqnarray*}
\partial(u)\wedge v&=&\sum_{1\le j< i}(-1)^{j+i} {[g_j, g_i]}\wedge  g_1\wedge\cdots\hat{ g}_j\cdots\wedge g_n\\
&=&\sum_{1\le j<i}(-1)^i g_1\wedge\cdots g_{j-1}\wedge {[g_j,g_i]}\wedge g_{j+1}\wedge\cdots\wedge g_n\\
&=&(-1)^i\mathrm{Tr}(\mathrm{ad}_{g_i}|_{\mathrm{span}\{g_1,\cdots, g_{i-1}\}})\cdot g_1\wedge\cdots\wedge g_n.
\end{eqnarray*}
Similarly, one may check
\begin{eqnarray*}
u\wedge \partial(v)&=&-\mathrm{Tr}(\mathrm{ad}_{g_i}|_{\mathrm{span}\{g_{i+1},
\cdots, g_n\}})\cdot g_1\wedge\cdots\wedge g_n.
\end{eqnarray*}
Thus $\langle\partial(u),v\rangle+(-1)^i\langle u,\partial(v)\rangle=(-1)^i\mathrm{Tr}(\mathrm{ad}_{g_i})$,
which is zero if and only if
$\mathfrak g$ is unimodular.
\end{proof}

Alternatively, $\mathfrak g$ is unimodular if and only if any nonzero
top chain $\Omega\in\mathrm C_n(\mathfrak g)$ is a Chevalley-Eilenberg cycle.
This is a chain version of the fact
that the cohomology of finite dimensional unimodular Lie algebras
admits Poincar\'e duality ({\it cf}. \cite[Chapter V]{GHV}).

\begin{theorem}[He-Van Oystaeyen-Zhang \cite{HVZ}]
Let $\mathfrak g$ be an $n$-dimensional
Lie algebra over $\k$.
Then $U(\mathfrak g)$ is $n$-Calabi-Yau if and only if $\mathfrak g$ is unimodular.
\end{theorem}

\begin{proof}
That $U(\mathfrak g)$ is homologically smooth follows from the fact that the Koszul
resolution of $U(\mathfrak g)$ is of length $n+1$.

By Proposition \ref{PDVdB} and Lemma \ref{PD_unimodular},
$$\mathrm C^{\bullet} (\mathfrak g)\simeq \mathrm C_{n-{\bullet}}(\mathfrak g)$$
implies
$$\mathrm{RHom}_{U(\mathfrak g)^e}(U(\mathfrak g), U(\mathfrak g)^e)\simeq\Sigma^{-n} U(\mathfrak g)$$
in $\mathcal D(U(\mathfrak g)^e)$, and vice versa.
\end{proof}

In particular, as we mentioned in \S1,
the space of polynomials $\k[x_1,x_2,\cdots,x_n]$ is $n$-Calabi-Yau.

\section{Proof of the main theorem}\label{Sect_Proof}

\subsection{Batalin-Vilkovisky algebras}

\begin{definition}[Batalin-Vilkovisky algebra]
Let $(V, {\hdot})$ be a graded commutative algebra.
A {\it Batalin-Vilkovisky operator} on $V$
is a degree $-1$ differential operator
$$
\Delta:V_{{\bullet}}\to V_{{\bullet}-1}
$$
such that the {\it deviation from being a derivation}
\begin{equation}\label{deviation}
[a, b]:=(-1)^{|a|+1}(\Delta(a{\bullet} b)-\Delta(a){\bullet} b-(-1)^{|a|}a{\bullet}\Delta(b)),\quad\mbox{for all}\,\, a, b\in V
\end{equation}
is a derivation for each component, {\it i.e.}
$$
[a, b{\bullet} c]=[a,b]{\bullet} c+(-1)^{|b|(|a|-1)}b{\bullet} [a, c],\quad\mbox{for all}\,\, a, b,c\in V.
$$
The triple $(V, {\bullet}, \Delta)$ is called a {\it Batalin-Vilkoviksy algebra}.
\end{definition}

For a graded associative algebra $V$, an linear operator $\Delta:V\to V$
(not necessarily a differential) satisfying
(\ref{deviation}) is said to be {\it of second order}.
Suppose $(V, {\hdot}, \Delta)$ is a Batalin-Vilkoviksy algebra.
Then by definition, $(V, {\hdot}, [-,-])$, where $[-,-]$ is given by (\ref{deviation}),
is a Gerstenhaber algebra. In other words, a Batalin-Vilkovisky algebra is a special
class of Gerstenhaber algebras.
Also $\Delta$ being of second order means
\begin{eqnarray}\label{second_order}
\Delta(a{\hdot} b{\hdot} c)&=&
\Delta(a{\hdot} b){\hdot} c+(-1)^{|b|\cdot |c|}\Delta(a{\hdot} c){\hdot} b
+(-1)^{|a|}a{\hdot}\Delta(b{\hdot} c)\nonumber\\
&&
-\Delta(a){\hdot} b{\hdot} c-(-1)^{|a|}a{\hdot}\Delta(b){\hdot} c-(-1)^{|a|+|b|}a{\hdot} b{\hdot} \Delta(c).
\end{eqnarray}
The reader may also refer to Getzler \cite{Getzler} for more details.

The following Theorems \ref{BV_CY} and \ref{BV_Symm} are due to Ginzburg and
Tradler respectively. Since each theorem has at least two proofs appeared in literature,
we will only sketch them in the following, just for reader's convenience.

\begin{theorem}[Ginzburg \cite{Ginzburg}, Theorem 3.4.3]\label{BV_CY}
Suppose that $A$ is an $n$-Calabi-Yau algebra.
Then the Hochschild cohomology $\mathrm{HH}^{\bullet}(A;A)$ has a Batalin-Vilkovisky algebra structure.
\end{theorem}

\begin{proof}[Sketch of proof]
The proof is a combination of the following three facts:
\begin{enumerate}
\item[$(1)$]
$\mathrm{HH}^{\bullet}(A; A)$ together with the Gerstenhaber cup product $\cup$
is a graded commutative algebra;
\item[$(2)$]
Via the noncommutative Poincar\'e duality
$\mathrm{PD}:\mathrm{HH}^{{\bullet}}(A; A)\stackrel{\simeq}\longrightarrow\mathrm{HH}_{n-{\bullet}}(A )$,
define a differential operator
$$\Delta:\mathrm{HH}^{{\bullet}}(A; A)
\longrightarrow\mathrm{HH}^{{\bullet}-1}(A;A)$$
by letting $\Delta:=\mathrm{PD}^{-1}\circ B\circ \mathrm{PD}$.
\item[$(3)$] $\Delta$ is a second order operator with respect to the Gerstenhaber cup product.
\end{enumerate}
In summary, $(\mathrm{HH}^{\bullet}(A;A),\cup,\Delta)$ is a Batalin-Vilkovisky algebra.
\end{proof}

The proof of Ginzburg uses the {\it Tamarkin-Tsygan calculus}, and we recommend the reader
to the original paper for more details; see also Lambre \cite{Lambre}.

\subsection{Hochschild (co)homology of cyclic algebras}
For cyclic algebras, we also have a version of Poincar\'e duality,
due to Tradler \cite{Tradler}:

\begin{lemma} \label{PD_Cyclic}
Let $A^!$ be a cyclic (not necessarily Koszul) algebra of degree $n$.
Denote by $A^{\ac}:=\mathrm{Hom}_{\k}(A,\k)$ its dual coalgebra.
Then there is an isomorphism
$$\mathrm{PD}:\mathrm{HH}^{{\bullet}}(A^!;A^! )
\stackrel{\simeq}\longrightarrow\mathrm{HH}_{n-{\bullet}}(A^{\ac}).$$
\end{lemma}

\begin{proof}First, we have an isomorphism of vector spaces
\begin{eqnarray}
\overline
{\mathrm{CH}}^{\bullet}(A^!;A^!)
&\simeq&\mathrm{Hom}^{\pi'}(\mathrm B(A^!), A^!)
\stackrel{(\ref{TransferingHomToTensorProduct})}{\simeq}A^!\ot^{\pi'} \mathrm{B}(A^!)^\vee\nonumber\\
&\stackrel{(\ref{DualityBetweenBarCobarConstructionForKoszulAlgebra})}{\simeq}&
A^!\ot^{\pi'}\Omega(A^\ac)
\stackrel{(\ref{DualTwistingMorphismsInducingIsomorphism})}{\simeq}
 \Omega(A^{\ac})\ot^{\iota} A^!
\nonumber\\
&\stackrel{(\ref{KD})}{\simeq} &\Sigma^{-n}( \Omega(A^{\ac})\ot A^{\ac}, b)
=\Sigma^{-n}\overline{\mathrm{CH}}_{\bullet}(A^{\ac}).\label{Iso_HochHoch}
\end{eqnarray}
We next show that
\begin{equation}\label{IsomorphismNeededInPDCyclic}\Omega(A^{\ac})\ot^{\iota}
A^!\stackrel{(\ref{KD})}{\simeq}\Sigma^{-n}( \Omega(A^{\ac})\ot A^{\ac} \end{equation}
 is an isomorphism of chain complexes.
Choose a basis $\{e_i\}$ for $ A^{\ac}$, and denote its dual by $\{e^i\}$.
The the differential on $\overline{\mathrm{CH}}^{\bullet}(A^!; A^!)\simeq\mathrm
\Omega(A^{\ac})\otimes^{\iota} A^!$
is given as follows:
for $ (a_1,\cdots, a_n , x)\in \mathrm \Omega(A^{\ac})\otimes A^!$,
\begin{eqnarray}
\delta( a_1,\cdots, a_n , x)&=&
(d(a_1,\cdots, a_n), x)
+\sum_i (-1)^{|a_1|+\cdots+|a_n|+|e_i|}\Big( a_1,\cdots, a_n, \overline e_i ,x\cdot e^i\Big)\nonumber\\
&&+\sum_i (-1)^{(|a_1|+\cdots+|a_n|)|e_i|}\Big(\overline e_i,a_1,\cdots, a_n , e^i\cdot x\Big).\label{HCohdiff}
\end{eqnarray}
where $d$ is the differential on the cobar construction, and $\overline e_i$
is the image of the projection $A^{\ac}\to\overline A^{\ac}=A^{\ac}/\k$.
Let $\psi: A^{!}\to \Sigma^{-n}A^{\ac}$ be the isomorphism (\ref{KD}),
then under (\ref{Iso_HochHoch}) the right hand of (\ref{HCohdiff}) is mapped
to
\begin{eqnarray}
&&(d(a_1,\cdots, a_n),\psi(x))
+\sum_i (-1)^{|a_1|+\cdots+|a_n|+|e_i|}\Big( a_1,\cdots, a_n, \overline e_i ,\psi(x\cdot e^i)\Big)\nonumber\\
&&+\sum_i (-1)^{(|a_1|+\cdots+|a_n|)|e_i|}\Big(\overline e_i,a_1,\cdots, a_n , \psi(e^i\cdot x)\Big).\label{HCohdiff1}
\end{eqnarray}
On the other hand, the Hochschild boundary on
$( a_1,\cdots, a_n , \psi(v))\in\overline{\mathrm{CH}}_{\bullet}(A^{\ac})$ is
\begin{eqnarray}
&&b( a_1,\cdots, a_n , \psi(x))\nonumber\\
&=&
(d(a_1,\cdots, a_n),\psi(x))
+\sum_{(\psi(x))} (-1)^{|a_1|+\cdots+|a_n|+|\psi(x)'|}
\Big( a_1,\cdots, a_n, \overline{\psi(x)}' ,\psi(x)''\Big)\nonumber\\
&&+\sum_{(\psi(x))} (-1)^{(|a_1|+\cdots+|a_n|)|\psi(x)''|}
\Big( \overline{\psi(x)}'',a_1,\cdots, a_n , \psi(x)'\Big).\label{HCohdiff2}
\end{eqnarray}
Therefore, to show (\ref{HCohdiff1}) equals the left hand of (\ref{HCohdiff2}),
it suffices to show that in $A^{\ac}\otimes A^{\ac}$,
\begin{equation}\label{HCohdiff55}
\sum_i e_i\otimes \psi(x e^i)=\sum_{(\psi(x))}\psi (x)'\otimes \psi(x)'',\quad
\sum_i e_i\otimes\psi(e^i x)=\sum_{(\psi(x))}\psi (x)''\otimes\psi (x)'.
\end{equation}
Pick two arbitrary basis $e^j, e^k\in A^!$, the evaluation
\begin{eqnarray}
\Big(\sum_i e_i\otimes \psi(x e^i)\Big)(e^j, e^k)&=&\sum_i \delta_i^j\cdot \psi(xe^i)(e^k)\nonumber\\
&=&\psi(xe^j)(e^k)\nonumber\\
&=&\psi(x)(e^j\cdot e^k),\label{HCohdiff3}
\end{eqnarray}
where the last equality holds since $\psi$ is a map of $A^!$-bimodules.
On the other hand,
\begin{eqnarray}
\Big(\sum_{(\psi(x))} \psi(x)'\otimes\psi(x)''\Big)(e^j, e^k)&=&\psi(x)(e^j\cdot e^k)\label{HCohdiff4}
\end{eqnarray}
automatically.
Comparing (\ref{HCohdiff3}) and (\ref{HCohdiff4}),
we obtain
the first equality of (\ref{HCohdiff55}). The second equality is proved similarly.
This completes the proof.
\end{proof}

\begin{theorem}[Tradler \cite{Tradler}, Theorem 1]\label{BV_Symm}
Let $A^!$ be a cyclic algebra and let $A^{\ac}$ be its dual coalgebra.
Then the Hochschild cohomology
$\mathrm{HH}^{\bullet}(A^!;A^!)$ has a Batalin-Vilkovisky algebra structure.
\end{theorem}

\begin{proof}[Sketch of proof]
The proof is also a combination of the following three facts:
\begin{enumerate}
\item[$(1)$]
$\mathrm{HH}^i(A^!;A^!)$ together with the Gerstenhaber cup product
is a graded commutative algebra;
\item[$(2)$]
Via the isomorphism
$\mathrm{PD}:\mathrm{HH}^{{\bullet}}(A^!; A^!)\stackrel{\simeq}\longrightarrow\mathrm{HH}_{n-{\bullet}}(A^{\ac})$,
we may define a differential operator
$$\Delta:\mathrm{HH}^{{\bullet}}(A^!; A^!)
\longrightarrow\mathrm{HH}^{{\bullet}-1}(A^!; A^!)$$
by letting $\Delta:=\mathrm{PD}^{-1}\circ \mathrm{B}\circ \mathrm{PD}$.
\item[$(3)$] $\Delta$ is a second order operator with respect to the Gerstenhaber cup product.
\end{enumerate}
In summary, $(\mathrm{HH}^{\bullet}(A^!;A^!),\cup,\Delta)$ is a Batalin-Vilkovisky algebra.
\end{proof}

The reader may also refer to Menichi \cite{Menichi} and Abbaspour \cite{Abbaspour}
for different proofs of this theorem.

\subsection{The main theorem}

Now we reach the main theorem of the current paper:

\begin{theorem}[Theorem \ref{Main Thm}]
Let $A$ be a Koszul $n$-Calabi-Yau algebra,
and let $A^!$ be its Koszul dual algebra.
Then there is an isomorphism
$$
\mathrm{HH}^{\bullet\bullet}(A; A)\simeq\mathrm{HH}^{\bullet\bullet}(A^{!}; A^!)
$$
of Batalin-Vilkovisky algebras.
\end{theorem}

\begin{proof}
Since $A$ is Koszul Calabi-Yau, we have the following commutative diagram
$$
\xymatrixcolsep{3.5pc}
\xymatrix{
\mathrm{CH}^{\bullet\bullet}(A; A)\ar[d]_-{\mathrm{Thm}.\;\ref{G-alg_pers}}^\simeq\ar[r]^-{\simeq}
&A\ot^\kappa A^!  \ar[r]^\simeq\ar[d]_{(\ref{IsomorphismOfComplexesInducedByKD})}^{\simeq}
&\Omega(A^\ac)\ot^{\iota}A^!\ar[d]^-{(\ref{IsomorphismNeededInPDCyclic})}_{\simeq}
&\mathrm{CH}^{\bullet\bullet}(A^!; A^!)\ar[l]_-{\simeq}\ar[d]^-{\mathrm{Lem}.\;\ref{PD_Cyclic}}_\simeq\\
\mathrm{CH}_{n-\bullet, n-\bullet} (A)\ar[r]^-{(\ref{HHo_Kos})}_\simeq&
\Sigma^{-n}(A\otimes A^{\ac}, b)\ar[r]^-{\simeq}&\Sigma^{-n}(\Omega(A^\ac)\otimes A^{\ac}, b)&
\mathrm{CH}_{n-\bullet, n-\bullet}(A^{\ac})\ar[l]_\simeq
}
$$
of chain complexes. In fact,
  the top line induces on Hochschild cohomology the isomorphisms of graded commutative algebras
(Theorem \ref{G-alg_pers}),
 the bottom line induces on Hochschild homology the isomorphisms of graded vector spaces which commutes
with $B$ (Theorem \ref{can_iso}), the leftmost square appeared in the proof of
Theorem~\ref{PDVdB} and the rightmost square in the proof of Lemma \ref{PD_Cyclic},
and the middle square is obviously commutative.
Combining with Theorems \ref{BV_CY} and \ref{BV_Symm}
the conclusion follows.
\end{proof}

\subsection{Remark}
In this paper we have only considered Koszul Calabi-Yau algebras. It is very likely that our arguments hold
for $N$-Koszul Calabi-Yau algebras in the sense of Berger (\cite{Berger}),
or more generally,
exact complete Calabi-Yau algebras in the sense of Van den Bergh (\cite{VdB3}).


\section{An application to cyclic homology}\label{Sect_App}

In mathematical literature, Batalin-Vilkovisky algebras are always related to deformation theory (the
Tian-Todorov Lemma).
However, the Batalin-Vilkovisky algebras on both sides of Theorem \ref{Main Thm} are not directly
(but indirectly) related
to the deformations of $A$ or $A^!$.
Indeed, due to the work of de Thanhoffer de V\"{o}lcsey and
Van den Bergh \cite{VdBdTdV12},
the deformations of $A$ are controlled by a DG Lie algebra whose homology is
the negative cyclic
homology $\mathrm{HC}_{\bullet}^{-}(A)$,
while the deformations of the cyclic algebra $A^{!}$
are controlled by the cyclic cohomology $\mathrm{HC}^{\bullet}(A^{!})$,
which is a work of Penkava and Schwarz \cite{Penkava}.
The theorem below, which is essentially a corollary of Theorem \ref{Main Thm},
gives an isomorphism of Lie algebras on these to cyclic (co)homology groups:

\begin{theorem}\label{Main_Thm2}
Let $A$ be a Koszul $n$-Calabi-Yau algebra, and let $A^!$ be its Koszul dual algebra.
Then there is an isomorphism
$$
\mathrm{HC}_{\bullet}^{-}(A)\simeq\mathrm{HC}^{-{\bullet}}(A^{!} )
$$
of Lie algebras between the negative cyclic homology of $A$ and the cyclic cohomology of $A^{!}$.
\end{theorem}

Consider the short exact sequence
$$0\longrightarrow
u\cdot \mathrm{CC}_{{\bullet}+2}^{-}(A)
\stackrel{\iota}\longrightarrow
\mathrm{CC}_{\bullet}^{-}(A)
\stackrel{\pi}\longrightarrow
\mathrm{CH}_{\bullet}(A)\longrightarrow 0,$$
where
$\iota: u\cdot \mathrm{CC}_{{\bullet}+2}^{-}(A)
\to
\mathrm{CC}_{\bullet}^{-}(A)
$ is the embedding
and
$$
\begin{array}{cccl}
\pi:&\mathrm{CC}_{\bullet}^{-}(A)&
\longrightarrow
&
\mathrm{CH}_{\bullet}(A)\\
&\displaystyle\sum_i x_i\cdot u^i&\longmapsto& x_0
\end{array}
$$
is the projection.
It induces a long exact sequence (observe that this is
the cohomological version of the Connes long exact sequence)
$$
\cdots\longrightarrow
\mathrm{HC}_{{\bullet}+2}^{-}(A)
\longrightarrow
\mathrm{HC}_{\bullet}^{-}(A)
\stackrel{\pi_*}\longrightarrow
\mathrm{HH}_{\bullet}(A)
\stackrel{\beta}\longrightarrow
\mathrm{HC}_{{\bullet}+1}^{-}(A)
\longrightarrow\cdots,
$$
where we observe that
$\mathrm{H}_{\bullet}(u \cdot \mathrm{CC}_{{\bullet}+2}^{-}(A))\simeq\mathrm{HC}_{{\bullet}+2}^{-}(A).$
It is obvious that $\beta\circ\pi_*=0$ and we claim that
$$\pi_*\circ\beta=B:\mathrm{HH}_{\bullet}(A)\to\mathrm{HH}_{{\bullet}+1}(A).$$
In fact, for any $x\in \mathrm{CH}_{\bullet}(A)$ which is $b$-closed,
from the following diagram
$$
\xymatrix{
&\ar[d]&\ar[d]&\ar[d]&\\
0\ar[r]
&u\cdot \mathrm{CC}_{\bullet}^{-}(A)\ar[r]^{\iota}\ar[d]^{b+uB}
&\mathrm{CC}_{\bullet}^{-}(A)\ar[d]^{b+uB}\ar[r]^{\pi}&\mathrm{CH}_{\bullet}(A)
\ar[d]^b\ar[r]&0\\
0\ar[r]
&u\cdot \mathrm{CC}_{{\bullet}-1}^{-}(A)\ar[r]^{\iota}\ar[d]^{b+uB}
&\mathrm{CC}_{{\bullet}-1}^{-}(A)\ar[d]^{b+uB}\ar[r]^{\pi}&\mathrm{CH}_{{\bullet}-1}(A)
\ar[d]^b\ar[r]&0\\
&&&&}
$$
we have (up to a boundary)
$$\iota^{-1}\circ(b+uB)\circ\pi^{-1}(x)=\iota^{-1}\circ(b+uB)(x)=\iota^{-1}(u\cdot B(x))=u\cdot B(x)
\in u\cdot\mathrm{CC}_{{\bullet}-1}^{-}(A).$$
Via the isomorphism
$u \cdot \mathrm{CC}_{\bullet}^{-}(A)\simeq\mathrm{CC}_{{\bullet}+2}^{-}(A)$,
this element $u\cdot B(x)$ is mapped to $B(x)\in\mathrm{CC}_{{\bullet}+1}^{-}(A)$,
and under $\pi$ it is mapped to $B(x)$.
Thus $\pi_*\circ\beta=B$ as claimed.

\begin{lemma}[\cite{VdBdTdV12}]\label{Lie1}
Suppose $A$ is an associative algebra. If $(\mathrm{HH}_{\bullet}(A), {\hdot}, B)$  is a
Batalin-Vilkovisky algebra, where ${\hdot}$ is a
graded commutative associative product,
then
$$\{a, b\}:=(-1)^{|a|}\beta(\pi_*(a){\hdot} \pi_*(b)), \quad\mbox{for homogeneous}\;\;
a, b\in\mathrm{HC}_{\bullet}^{-}(A)$$
defines a degree one graded Lie algebra structure on $\mathrm{HC}_{\bullet}^{-}(A)$.
\end{lemma}

\begin{proof}
We first show the graded skew-symmetry.
In fact, for two homogeneous $a, b\in\mathrm{HC}_{\bullet}^{-}(A)$,
\begin{eqnarray*}
&&\{a,b\}+(-1)^{(|a|+1)(|b|+1)}\{b,a\}\\
&=&(-1)^{|a|}\beta(\pi_*(a){\hdot}\pi_*(b))+(-1)^{(|a|+1)(|b|+1)+|b|}
\beta(\pi_*(b){\hdot}\pi_*(a))\\
&=&(-1)^{|a|}\beta(\pi_*(a){\hdot}\pi_*(b))+(-1)^{(|a|+1)(|b|+1)+|b|+|a|\cdot |b|}
\beta(\pi_*(a){\hdot}\pi_*(b))\\
&=&0.
\end{eqnarray*}
Next, we show graded Jacobi identity:
for homogeneous $a,b,c\in\mathrm{HC}_{\bullet}^{-}(A)$,
\begin{eqnarray*}
\{\{a,b\},c\}&=&(-1)^{|b|+1}\beta(\pi_*(\beta(\pi_*(a){\hdot}\pi_*(b))){\hdot}\pi_*(c))\\
&=&(-1)^{|b|+1}\beta(B(\pi_*(a){\hdot}\pi_*(b)){\hdot}\pi_*(c)).
\end{eqnarray*}
Similarly,
\begin{eqnarray*}
\{\{c,a\},b\}&=&(-1)^{|a|+1}\beta(B(\pi_*(c){\hdot}\pi_*(a)){\hdot}\pi_*(b)),\\
\{\{b,c\},a\}&=&(-1)^{|c|+1}\beta(B(\pi_*(b){\hdot}\pi_*(c)){\hdot}\pi_*(a)).
\end{eqnarray*}
Now since $(\mathrm{HH}_{\bullet}(A),{\hdot},B)$ is a Batalin-Vilkovisky algebra,
by (\ref{second_order}) we obtain
\begin{eqnarray*}
&&(-1)^{(|a|+1)(|c|+1)}\{\{a,b\},c\}+(-1)^{(|c|+1)(|b|+1)}\{\{c,a\},b\}+(-1)^{(|b|+1)(|a|+1)}\{\{b,c\},a\}\\
&=&(-1)^{(|a|+1)(|c|+1)+|b|+1}\beta(B(\pi_*(a){\hdot}\pi_*(b)){\hdot}\pi_*(c))\\
&&+
(-1)^{(|c|+1)(|b|+1)+|a|+1}\beta(B(\pi_*(c){\hdot}\pi_*(a)){\hdot}\pi_*(b))\\
&&
+(-1)^{(|b|+1)(|a|+1)+|c|+1}\beta(B(\pi_*(b){\hdot}\pi_*(c)){\hdot}\pi_*(a))\\
&\stackrel{(\ref{second_order})}=
&(-1)^{|a|+|b|+|c|+|a|\cdot|c|}\Big(
\beta(B(\pi_*(a){\hdot}\pi_*(b){\hdot}\pi_*(c)))-\beta(B\circ\pi_*(a){\hdot}\pi_*(b){\hdot}\pi_*(c))\\
&&
-(-1)^{|a|}\beta(\pi_*(a){\hdot} B\circ\pi_*(b){\hdot}\pi_*(c))
-(-1)^{|a|+|b|}\beta(\pi_*(a){\hdot}\pi_*(b){\hdot} B\circ\pi_*(c))\Big)\\
&=&0,
\end{eqnarray*}
where the last equality holds since $B=\pi_*\circ \beta$ and therefore $\beta\circ B=B\circ\pi_*=0$.
\end{proof}

Similarly from the Connes long exact sequence for cyclic cohomology
$$
\cdots\longrightarrow
\mathrm{HH}^{{\bullet}}(A;\k)\stackrel{\beta}\longrightarrow
\mathrm{HC}^{{\bullet}-1}(A)\longrightarrow
\mathrm{HC}^{{\bullet}+1}(A)\stackrel{\pi_*}\longrightarrow
\mathrm{HH}^{{\bullet}+1}(A;\k)\longrightarrow\cdots,
$$
we have the following lemma, for which we omit
the proof:

\begin{lemma}\label{Lie2}
Suppose $A^{!}$ is an associative algebra. Denote by $\mathrm{HH}^{\bullet}(A^{!};\k)$
be the Hochschild
cohomology of $A^{!}$. Suppose $(\mathrm{HH}^{\bullet}(A^{!};\k), {\hdot},B^*)$
is a Batalin-Vilkovisky algebra, where ${\hdot}$ is a graded commutative associative product, then
$\mathrm{HC}^{\bullet}(A^{!})$ has a degree one Lie algebra structure, where
for any $x, y\in\mathrm{HC}^{\bullet}(A^!)$,
$$\{x,y\}:=(-1)^{|x|}\beta(\pi_*(x){\hdot}\pi_*(y)).$$
\end{lemma}

The key point of the above two lemmas is that there is no a priori
graded commutative associative product on $\mathrm{HH}_{\bullet}(A)$
and/or $\mathrm{HH}^{\bullet}(A^!;\k)$.
If $A$ is Calabi-Yau, then $\mathrm{HH}_{\bullet}(A)$ has a product induced
from the Gerstenhaber product on $\mathrm{HH}^{\bullet}(A;A)$ via Van den Bergh's
Poincar\'e duality, and similarly, if $A^!$ is a cyclic associative algebra,
then $\mathrm{HH}^{\bullet}(A^!;\k)$ has a product induced from
the Gerstenhaber product on $\mathrm{HH}^{\bullet}(A^!; A^!)$ via
Tradler's isomorphism.
However, during this process one has to notice that
the induced product has a degree.
For example, since the Gerstenhaber product on $\mathrm{HH}^{\bullet}(A;A)$ is
of degree zero, the induced product on $\mathrm{HH}_{\bullet}(A)$ has degree $-n$.
Therefore we have to shift the degree of $\mathrm{HH}_{\bullet}(A)$ down by $-n$
to get a Batalin-Vilkovisky algebra.
Let us take this degree shift in the following.

\begin{proof}[Proof of Theorem \ref{Main_Thm2}]
Since $A$ is Calabi-Yau,
$\mathrm{HH}_{\bullet}(A)$ admits a Batalin-Vilkovisky algebra structure
via Van den Bergh's duality
$$\mathrm{HH}_{\bullet}(A)\simeq\mathrm{HH}^{n-{\bullet}}(A; A).$$
The Lie bracket on $\mathrm{HC}_{\bullet}^{-}(A)$ is given
by Lemma \ref{Lie1} (see \cite[Theorem 10.2]{VdBdTdV12}),
while the Lie bracket on $\mathrm{HC}^{\bullet}(A^!)$ is given
by Lemma \ref{Lie2} (see a brief discussion in \cite[\S8]{Penkava}).

We have shown that $\mathrm{HH}^{\bullet}(A; A)$ and $\mathrm{HH}^{\bullet}(A^!; A^!)$
are isomorphic as Batalin-Vilkovisky algebras,
and therefore to show the Lie algebra isomorphism, we have to compare $\beta$ and $\pi_*$.
In fact,
by Lemma \ref{Main_Lemma2} and Proposition
\ref{Alg_Coalg}, we have a quasi-isomorphism
of mixed complexes
$$(\mathrm{HC}_{\bullet}(A), b, B)\simeq(\mathrm{HC}_{\bullet}(A^{\ac}), b, B)
\simeq(\mathrm{HC}^{-{\bullet}}(A^!), \delta, B^*),$$
and therefore a commutative long exact sequence
$$
\xymatrix{
\cdots\ar[r]&\mathrm{HH}_{\bullet}(A)\ar[d]\ar[r]^{\beta} &\mathrm{HC}_{{\bullet}+1}^{-}(A)\ar[r] \ar[d]
&\mathrm{HC}_{{\bullet}-1}^{-}(A)\ar[r]^{\pi_*} \ar[d]&
\mathrm{HH}_{{\bullet}-1}(A)\ar[r]\ar[d] &\cdots
\\
\cdots\ar[r]&\mathrm{HH}^{-{\bullet}}(A^!;\k)\ar[r]^{\beta} &\mathrm{HC}^{-{\bullet}-1}(A^!)\ar[r]
&\mathrm{HC}^{-{\bullet}+1}(A^!)\ar[r]^{\pi_*} &
\mathrm{HH}^{-{\bullet}+1}(A^!;\k)\ar[r]&\cdots}
$$
where the vertical maps are all isomorphism.
This completes the proof.
\end{proof}

\begin{remark}
As been observed in \cite{VdBdTdV12},
the above Lemmas \ref{Lie1} and \ref{Lie2} are very much similar to the ones given by Menichi
\cite{Menichi}, which has its precursor in string topology \cite{CS}.
However, there is a slight difference between them, especially
the degree of the bracket in the above lemmas is one, while the degree of the one of Menichi
is two.

To compare them, let us briefly recall the construction of Menichi.
From the homological version of the Connes long exact sequence
$$\cdots\stackrel{\mathrm M}\longrightarrow
\mathrm{HH}_{\bullet}(A)\stackrel{\mathrm E}\longrightarrow \mathrm{HC}_{\bullet}(A)
\longrightarrow
\mathrm{HC}_{{\bullet}-2}(A)
\stackrel{\mathrm M}\longrightarrow
\mathrm{HH}_{{\bullet}-1}(A)\stackrel{\mathrm E}\longrightarrow\cdots,
$$
again, we have that
$B=\mathrm{M}\circ\mathrm E: \mathrm{HH}_{\bullet}(A)\to\mathrm{HH}_{{\bullet}+1}(A)$.
The following is proved by Menichi \cite[Proposition 28]{Menichi},
with the same argument as in Lemma \ref{Lie1}.

\begin{theorem}[Menichi]
Suppose $A$ is an associative algebra.
If $(\mathrm{HH}_{\bullet}(A), \hdot, B)$  is a
Batalin-Vilkovisky algebra, where ${\hdot}$ is a
graded commutative associative product,
then
$$\{a, b\}:=(-1)^{|a|+1}
\mathrm E(\mathrm M(a){\hdot} \mathrm M(b)), \quad\mbox{for}\;\; a, b\in\mathrm{HC}_{\bullet}(A)$$
defines a degree two graded Lie algebra structure on $\mathrm{HC}_{\bullet}(A)$.
\end{theorem}

This Lie algebra is also interesting, since it gives an algebraic interpretation of the Lie algebra on
the $S^1$-equivariant homology of the free loop space of a compact smooth
manifold, discovered in string topology \cite{CS}.
\end{remark}


\end{document}